\documentclass[a4paper,11pt,reqno]{amsart}

\usepackage{amssymb}

\newtheorem{theorem}{Theorem}
\newtheorem{lemma}[theorem]{Lemma}
\newtheorem{corollary}[theorem]{Corollary}
\newtheorem{proposition}[theorem]{Proposition}

\numberwithin{equation}{section}
\numberwithin{theorem}{section}

\theoremstyle{definition}
\newtheorem{definition}[theorem]{Definition}
\newtheorem*{example*}{Example}

\newtheorem{remark}[theorem]{Remark}
\newtheorem*{remark*}{Remark}

\newcommand{\bF}{{\mathbb F}}
\newcommand{\bZ}{{\mathbb Z}}

\newcommand{\cT}{{\mathcal T}}
\newcommand{\cL}{{\mathcal L}}
\newcommand{\frg}{{\mathfrak g}}

\newcommand{\frf}{{\mathfrak f}}

\newcommand{\calS}{{\mathcal S}}

\newcommand{\calT}{{\mathcal T}}
\newcommand{\calL}{{\mathcal L}}

\newcommand{\subo}{_{\bar 0}}
\newcommand{\subuno}{_{\bar 1}}

\providecommand{\espan}[1]{\text{span}\left\{ #1\right\}}

\providecommand{\ptriple}[1]{\boldsymbol{(}#1\boldsymbol{)}}

 \DeclareMathOperator{\frsl}{{\mathfrak{sl}}}

 \DeclareMathOperator{\frd}{{\mathfrak{d}}}

 \DeclareMathOperator{\trace}{trace}

  \DeclareMathOperator{\Centr}{Centr}
  
 \DeclareMathOperator{\der}{\mathfrak{der}}
 
 \DeclareMathOperator{\End}{End}
 \DeclareMathOperator{\Hom}{Hom}
 \DeclareMathOperator{\Mat}{Mat}
 \DeclareMathOperator{\Aut}{Aut}
 
 \DeclareMathOperator{\Cent}{Cent}

  \DeclareMathOperator{\Dic}{Dic}

\newenvironment{romanenumerate}
 {\begin{enumerate}
 \renewcommand{\itemsep}{4pt}
 }{\end{enumerate}}

\begin{document}

\title[Special Freudenthal-Kantor triple systems]{Special Freudenthal-Kantor triple systems and Lie algebras with dicyclic symmetry}

\author[Alberto Elduque]{Alberto Elduque$^{\star}$}
 \thanks{$^{\star}$ Supported by the Spanish Ministerio de
 Educaci\'{o}n y Ciencia
 and FEDER (MTM 2007-67884-C04-02) and by the
Diputaci\'on General de Arag\'on (Grupo de Investigaci\'on de
\'Algebra)}
 \address{Departamento de Matem\'aticas e
 Instituto Universitario de Matem\'aticas y Aplicaciones,
 Universidad de Zaragoza, 50009 Zaragoza, Spain}
 \email{elduque@unizar.es}

\author[Susumu Okubo]{Susumu Okubo$^{\star\star}$}
 \thanks{$^{\star\star}$ Supported in part by U.S.~Department of Energy Grant No. DE-FG02-91 ER40685.}
 \address{Department of Physics and Astronomy, University of
 Rochester, Rochester, NY 14627, USA}
 \email{okubo@pas.rochester.edu}


\date{April 7, 2010}



\begin{abstract}
Lie algebras endowed with an action by automorphisms of the dicyclic group of degree $3$ are considered. The close connections of these algebras with Lie algebras graded over the nonreduced root system $BC_1$, with $J$-ternary algebras and with Freudenthal-Kantor triple systems are explored.
\end{abstract}

\maketitle

\section{Introduction}

Gradings of Lie algebras by abelian groups constitute a classical subject (see \cite{KochetovSurvey} for a survey of results). In characteristic $0$, any grading by a finitely generated abelian group is given by an action by automorphisms of an abelian group (the group of characters) on the Lie algebra.

Recently the authors have considered actions of some small nonabelian groups by automorphisms on a Lie algebra. More precisely, and motivated by the Principle of Triality that appears related to the classical Lie algebra of type $D_4$, Lie algebras with an action of the symmetric groups $S_3$ and $S_4$ by automorphisms were considered in \cite{EOS4}, \cite{EldTetra}, \cite{EOS4Tits}, \cite{EOS3S4}. It turns out that Lie algebras with an action of $S_4$ by automorphisms are coordinatized by some nonassociative algebras \cite{EOS4}. The unital algebras in this class are precisely the structurable algebras of Allison \cite{All78}. The Lie algebras with an action of $S_3$ by automorphisms are coordinatized by a class of algebras which generalize the Malcev algebras \cite{EOS3S4}. These are some binary-ternary algebras termed \emph{generalized Malcev algebras}. Jordan triple systems are instances of these generalized Malcev algebras.

As proved in \cite[Section 7]{EOS4Tits}, the Lie algebras graded by the nonreduced root system $BC_1$ of type $B_1$ are naturally among the Lie algebras with $S_4$-symmetry (that is, with an action of $S_4$ by automorphisms).

\smallskip

The Lie algebras graded by the nonreduced root system $BC_1$ but of type $C_1$ present another type of symmetry given by the dicyclic group
\[
\Dic_3=\langle \theta,\phi: \theta^4=1=\phi^3,\ \phi\theta\phi=\theta\rangle.
\]

Our aim in this paper is to study those Lie algebras with such a symmetry.

It will be shown that Lie algebras with $\Dic_3$-symmetry are strongly related to the $BC_1$-graded Lie algebras of type $C_1$, to the class of $J$-ternary algebras introduced in \cite{Allison76}, and to the Freudenthal-Kantor triple systems defined in \cite{YO84}. Moreover, these Lie algebras are coordinatized by elements in a class of nonassociative algebraic systems consisting of binary-ternary algebras endowed with an involutive automorphism which will be introduced here and named \emph{dicyclic ternary algebras}.

\smallskip

In Section \ref{se:BC1C1Jternary} the Lie algebras graded over the nonreduced root system $BC_1$ of type $C_1$ will be reviewed, as well as their connections with $J$-ternary algebras. These $J$-ternary algebras are closely connected with a subclass of the Freudenthal-Kantor triple systems introduced by Yamaguti and Ono in \cite{YO84}. This subclass will be defined in Section \ref{se:JternarySpecialFKTS} and has its own independent interest.

Most of the arguments can be superized and a class of \emph{strictly $BC_1$-graded Lie superalgebras of type $C_1$} will be considered in Section \ref{se:super}, as well as the associated $J$-ternary $(-1)$-algebras and $(-1,-1)$ Freudenthal-Kantor triple systems.

Section \ref{se:dicyclic} will introduce the \emph{dicyclic ternary algebras}, which coordinatize the Lie algebras with dicyclic symmetry. The dicyclic ternary algebras attached to $J$-ternary algebras will be the object of study of Section \ref{se:dicyclicJternary}. Finally, Section \ref{se:edFKTSdicyclic} will be devoted to show that the $5$-graded Lie algebra (or superalgebra)  attached to any $(\epsilon,\delta)$ Freudenthal-Kantor triple system (not necessarily special) is always endowed with an action of $\Dic_3$ by automorphisms. For $\delta=-\epsilon$ this action reduces to an action of the symmetric group $S_3$. For $\epsilon=\delta=1$ we get Lie algebras with $\Dic_3$-symmetry whose attached dicyclic ternary algebra is computed.

\smallskip

Throughout the paper all the algebraic systems considered will be defined over a ground field $\bF$ whose characteristic will always be assumed to be $\ne 2,3$.

\medskip

\section{$BC_1$-graded Lie algebras of type $C_1$\\ and $J$-ternary algebras}\label{se:BC1C1Jternary}

The purpose of this section is to review the relationship between Lie algebras graded over the non reduced root system $BC_1$ of type $C_1$, and a class of nonassociative algebraic systems named $J$-ternary algebras. This goes back to \cite{Allison76} (see \cite{Hein75} and \cite{Kan73} for closely related work). It must be mentioned that the Lie algebras graded by the root system $BC_1$ (of any type) are characterized in \cite{BenkartSmirnov03}, which completes the work in \cite{AllisonBenkartGao02},

The proof given here is different from the one in \cite{Allison76}.

\smallskip

Let $\frg$ be a Lie algebra over  $\bF$ containing a subalgebra isomorphic to $\frsl(V)$, for a two-dimensional vector space $V$ over $\bF$ (that is, a subalgebra isomorphic to $\frsl_2(\bF)$), and such that, as a module for $\frsl(V)$, the Lie algebra $\frg$ is a direct sum of modules which are either trivial, or isomorphic to $V$, or to the adjoint module $\frsl(V)$. These are precisely the required conditions for $\frg$ to be a $BC_1$-graded Lie algebra of type $C_1$.

Then, as a module for $\frsl(V)$, we may write:
\begin{equation}\label{eq:gslVVd}
\frg=\bigl(\frsl(V)\otimes J\bigr)\oplus\bigl(V\otimes T\bigr)\oplus\frd,
\end{equation}
where $\frd$ s the sum of the trivial modules (this is the centralizer in $\frg$ of the subalgebra $\frsl(V)=\frsl(V)\otimes 1$, and hence it is a subalgebra), and $J$ and $T$ are vector spaces, with $J$ containing a distinguished element $1$, such that $\frsl(V)\otimes 1$ is the subalgebra $\frsl(V)$ we have started with.

\smallskip

In the following lemma we collect some well-known results. Unadorned tensor products are always considered over the ground field $\bF$.

\begin{lemma}\label{le:invariantV}
Let $V$ be a two-dimensional vector space.
\begin{romanenumerate}
\item
The skew symmetric map $\frsl(V)\otimes\frsl(V)\rightarrow \frsl(V)$, $f\otimes g\mapsto [f,g]=fg-gf$, spans the space of $\frsl(V)$-invariant maps $\Hom_{\frsl(V)}(\frsl(V)\otimes\frsl(V),\frsl(V))$.

\item The symmetric map $\frsl(V)\otimes \frsl(V)\rightarrow \bF$, $f\otimes g\mapsto \trace(fg)$, spans the space of $\frsl(V)$-invariant maps $\Hom_{\frsl(V)}(\frsl(V)\otimes\frsl(V),\bF)$.

\item For any $f,g,h\in \frsl(V)$, one has
\begin{equation}\label{eq:[[fg]h]}
[[f,g],h]=2\bigl(\trace(gh)f-\trace(fh)g\bigr).
\end{equation}

\item The map $\frsl(V)\otimes V\rightarrow V$, $f\otimes u\mapsto f(u)$, spans the space of $\frsl(V)$-invariant maps $\Hom_{\frsl(V)}(\frsl(V)\otimes V,V)$.

\item The space $\Hom_{\frsl(V)}(V\otimes V,\bF)$ of $\frsl(V)$-invariant maps from $V\otimes V$ into the trivial module $\bF$ is one-dimensional and spanned by a nonzero skew symmetric bilinear map $\bigl(u\vert v\bigr)$.

\item The symmetric map $V\otimes V\rightarrow \frsl(V)$, $u\otimes v\mapsto \gamma_{u,v}=\bigl(u\vert .\bigr)v+\bigl(v\vert .\bigr)u$ ($(.\vert.)$ as in \textup{(v)}), spans the space of $\frsl(V)$-invariant maps $\Hom_{\frsl(V)}(V\otimes V,\frsl(V))$.

\item There are no nontrivial $\frsl(V)$-invariant maps $\frsl(V)\otimes\frsl(V)\rightarrow V$ and $V\otimes V\rightarrow V$. That is, $\Hom_{\frsl(V)}(\frsl(V)\otimes \frsl(V),V)=0=\Hom_{\frsl(V)}(V\otimes V,V)$. \qed
\end{romanenumerate}
\end{lemma}

Note that for any $f\in\frsl(V)$, $f^2=\frac{1}{2}\trace(f^2)1$, and hence $fg+gf=\trace(fg)1$ for any $f,g\in\frsl(V)$. This gives $fgf=\trace(fg)f-\frac{1}{2}\trace(f^2)g$, so $fgh+hgf=\trace(fg)h+\trace(gh)f-\trace(fh)g$, and finally $[[f,g],h]=(fgh+hgf)-(gfh+hfg)=2\bigl(\trace(gh)f-\trace(fh)g\bigr)$ for any $f,g\in \frsl(V)$, thus getting \eqref{eq:[[fg]h]}.

\medskip

From now on, we will fix a nonzero skew symmetric bilinear form $(.\vert .)$ on our two-dimensional vector space $V$.

The $\frsl(V)$-invariance of the Lie bracket in our Lie algebra $\frg$ in \eqref{eq:gslVVd} gives, for any $f,g\in\frsl(V)$, $u,v\in V$ and $\varphi\in \frd$, the following conditions:
\begin{equation}\label{eq:g[]}
\begin{split}
[f\otimes a,g\otimes b]&=[f,g]\otimes a\cdot b+2\trace(fg)D_{a,b},\\
[f\otimes a,u\otimes x]&=f(u)\otimes a\bullet x,\\
[u\otimes x,v\otimes y]&=\gamma_{u,v}\otimes \langle x\vert y\rangle +\bigl(u\vert v\bigr)d_{x,y},\\
[\varphi,f\otimes a]&=f\otimes \varphi(a),\\
[\varphi,u\otimes x]&=u\otimes \varphi(x),
\end{split}
\end{equation}
for suitable bilinear maps
\begin{equation}\label{eq:maps}
\begin{split}
J\times J\rightarrow J,&\quad (a,b)\mapsto a\cdot b,\quad\text{(symmetric),}\\
J\times J\rightarrow \frd,&\quad (a,b)\mapsto D_{a,b},\quad\text{(skew symmetric),}\\
J\times T\rightarrow T,&\quad (a,x)\mapsto a\bullet x,\\
T\times T\rightarrow J,&\quad (x,y)\mapsto \langle x\vert y\rangle\quad\text{(skew symmetric),}\\
T\times T\rightarrow \frd,&\quad (x,y)\mapsto d_{x,y}\quad\text{(symmetric),}\\
\frd\times J\rightarrow J,&\quad (\varphi,a)\mapsto \varphi(a),\\
\frd\times T\rightarrow T,&\quad (\varphi,x)\mapsto \varphi(x),
\end{split}
\end{equation}
such that $1\cdot a=a$, $D_{1,a}=0$ and $1\bullet x=x$ for any $a\in J$ and $x\in T$.

The Jacobi identity on $\frg$ shows that all these maps are invariant under the action of the Lie subalgebra $\frd$. It also gives some other conditions to be satisfied by these maps:

\begin{itemize}
\addtolength{\itemindent}{-20pt}
\addtolength{\itemsep}{8pt}
\item For any $f,g,h\in \frsl(V)$ and $a,b,c\in J$ we have,
\[
\begin{split}
[[f\otimes a,g\otimes b],h\otimes c]
&=[[f,g],h]\otimes (a\cdot b)\cdot c\\
&\qquad+2\trace(fg)h\otimes D_{a,b}(c)+2\trace([f,g]h)D_{a\cdot b,c}
\end{split}
\]
For linearly independent $f,g,h$ (generic case), equation \eqref{eq:[[fg]h]} implies that $\sum_{\text{cyclic}}[[f\otimes a,g\otimes b],h\otimes c]$ equals $0$ if and only if the following two conditions hold:
\begin{equation}\label{eq:Ds}
\begin{split}
&D_{a\cdot b,c}+D_{b.c,a}+D_{c\cdot a,b}=0,\\
&D_{a,b}(c)=a\cdot(b\cdot c)-b\cdot(a\cdot c ),
\end{split}
\end{equation}
for any $a,b,c\in J$. Since the characteristic is $\ne 3$, the first equation in \eqref{eq:Ds} is equivalent to $D_{a\cdot a,a}=0$, which in view of the second equation is equivalent to:
\[
(a\cdot a)\cdot (b\cdot a)=((a\cdot a)\cdot b)\cdot a
\]
for any $a,b\in J$. That is, $J$ is a Jordan algebra.

This goes back to \cite{Tits62}, where the case in which $T=0$ is considered and the well-known Tits-Kantor-Koecher construction of Lie algebras out of Jordan algebras appeared for the first time (see also \cite{Jac67}).

Therefore the Jacobi identity for elements in $\frsl(V)\otimes J$ is equivalent to $(J,\cdot)$ being a Jordan algebra and to the condition $D_{a,b}(c)=a\cdot(b\cdot c)-b\cdot(a\cdot c)$ for any $a,b,c\in J$ (that is, $D_{a,b}$ acts as the natural inner derivation of $J$ attached to the elements $a,b\in J$ \cite{Jac67}).

\item For any $f,g\in\frsl(V)$, $a,b\in J$, $w\in V$ and $x\in T$ we get:
\[
\begin{split}
[[f\otimes a,g\otimes b],w\otimes x]&=[[f,g]\otimes a\cdot b +2\trace(fg)D_{a,b},w\otimes x]\\
&=[f,g](w)\otimes (a\cdot b)\bullet x +2\trace(fg)w\otimes D_{a,b}(x),\\[6pt]
[[f\otimes a,w\otimes x],g\otimes b]&=
[f(w)\otimes a\bullet x,g\otimes b]=-g(f(w))\otimes b\bullet(a\bullet x),\\
[f\otimes a,[g\otimes b,w\otimes x]]&=[f\otimes a,g(w)\otimes b\bullet x]=f(g(w))\otimes a\bullet(b\bullet x).
\end{split}
\]
With $f=g$, and since $f^2=\frac{1}{2}\trace(f^2)1$, Jacobi identity gives:
\begin{equation}\label{eq:Dabx}
4D_{a,b}(x)=a\bullet (b\bullet x)-b\bullet(a\bullet x).
\end{equation}
Now, using this fact and since $fg+gf=\trace(fg)1$, with $f,g\in\frsl(V)$ and $w\in W$ such that $f(g(w))$ and $g(f(w))$ are linearly independent (generic case), we get
\begin{equation}\label{eq:abx}
(a\cdot b)\bullet x=\frac{1}{2}\Bigl(a\bullet(b\bullet x)+b\bullet(a\bullet x)\Bigr).
\end{equation}
That is, the map $J\rightarrow \End_\bF(T)$, $a\mapsto \lambda_a$, with $\lambda_a(x)= a\bullet x$, is a homomorphism of Jordan algebra, where $\End_\bF(T)$ is a Jordan algebra under the product $\alpha\circ\beta=\frac{1}{2}\bigl(\alpha\beta+\beta\alpha)$. This (special) Jordan algebra is denoted by $\End_{\bF}(T)^+$. In other words, $T$ is a \emph{special unital Jordan module} for $J$.

Conversely, equations \eqref{eq:Dabx} and \eqref{eq:abx} imply the Jacobi identity for the elements above.

\item For any $f\in\frsl(V)$, $a\in J$, $u,v\in V$, and $x,y\in T$, we get:
\[
\begin{split}
[f\otimes a,[u\otimes x,v\otimes y]]&=[f\otimes a,\gamma_{u,v}\otimes \langle x\vert y\rangle +(u\vert v)d_{x,y}]\\
&=[f,\gamma_{u,v}]\otimes a\cdot \langle x\vert y\rangle -(u\vert v)f\otimes d_{x,y}(a)\\
&\qquad\qquad +2\trace(f\gamma_{u,v})D_{a,\langle x\vert y\rangle},\\[6pt]
[[f\otimes a,u\otimes x],v\otimes y]&=[f(u)\otimes a\bullet x,v\otimes y]\\
&=\gamma_{f(u),v}\otimes \langle a\bullet x\vert y\rangle +(f(u)\vert v)d_{a\bullet x,y},\\
[u\otimes x,[f\otimes a,v\otimes y]]&=[u\otimes x,f(v)\otimes a\bullet y]\\
& =\gamma_{u,f(v)}\otimes \langle x\vert a\bullet y\rangle +(u\vert f(v))d_{x,a\bullet y}.
\end{split}
\]
But $\trace(f\gamma_{u,v})=\trace\bigl(f\bigl((u\vert .)v+(v\vert .)u\bigr)\bigr)=(u\vert f(v))+(v\vert f(u))=-2\bigl(f(u)\vert v)=2(u\vert f(v))$, so the component in $\frd$ of the above equations give:
\begin{equation}\label{eq:Daxy}
4D_{a,\langle x\vert y\rangle}=-d_{a\bullet x,y}+d_{x,a\bullet y},
\end{equation}
for any $a\in J$ and $x,y\in T$.

Also $[f,\gamma_{u,v}]=\gamma_{f(u),v}+\gamma_{u,f(v)}$. Hence with $u=v$, $[f,\gamma_{u,u}]=2\gamma_{f(u),u}=2\gamma_{u,f(u)}$, and we get
\begin{equation}\label{eq:axy}
2a\cdot\langle x\vert y\rangle=\langle a\bullet x\vert y\rangle +\langle x\vert a\bullet y\rangle,
\end{equation}
for any $a\in J$ and $x,y\in T$.

Since the dimension of $V$ is $2$, the skew symmetry forces
\[
\sum_{\text{cyclic}}(w_1\vert w_2)w_3=0.
\]
Hence it follows that for any $u,v,w\in V$, $(u\vert v)f(w)+(f(w)\vert u)v+(v\vert f(w))u=0$, or
\begin{equation}\label{eq:uvf}
(u\vert v)f=-(f(u)\vert .)v+(f(v)\vert .)u,
\end{equation}
for any $u,v\in V$ and $f\in\frsl(V)$. Now, take $u$ and $v$ linearly independent in $V$, and take $f=\gamma_{u,v}$. Then $[f,\gamma_{u,v}]=0$, $\gamma_{f(u),v}=-(u\vert v)f=-\gamma_{u,f(v)}$, so we obtain:
\begin{equation}\label{eq:dxya}
d_{x,y}(a)=\langle a\bullet x\vert y\rangle -\langle x\vert a\bullet y\rangle.
\end{equation}

Conversely, equations \eqref{eq:Daxy}, \eqref{eq:axy} and \eqref{eq:dxya} imply the Jacobi identity for the elements above.

\item Finally, for $u,v,w\in V$ and $x,y,z\in T$,
\[
\begin{split}
[[u\otimes x,v\otimes y],w\otimes z]&=[\gamma_{u,v}\otimes \langle x\vert y\rangle +(u\vert v)d_{x,y},w\otimes z]\\
&=\gamma_{u,v}(w)\otimes \langle x\vert y\rangle\bullet z +(u\vert v)w\otimes d_{x,y}(z).
\end{split}
\]
With $u$ and $v$ such that $(u\vert v)=1$ and $w=u$, $\gamma_{u,v}(u)=-u=\gamma_{v,u}(u)$, $\gamma_{u,u}(v)=2v$ and the Jacobi identity gives:
\begin{equation}\label{eq:dxyzdzyx}
d_{x,y}(z)-d_{z,y}(x)=\langle x\vert y\rangle\bullet z-\langle z\vert y\rangle \bullet x+2\langle x\vert z\rangle \bullet y
\end{equation}
for any $x,y,z\in T$. And conversely, this equation is sufficient to obtain the validity of the Jacobi identity for elements in $V\otimes T$.
\end{itemize}

\medskip

Let us collect the information obtained so far in the following result:

\begin{theorem}\label{th:gBC1}
Let $\frg$ be a Lie algebra. Then $\frg$ is $BC_1$-graded of type $C_1$ if and only if there is a two-dimensional vector space $V$ such that that $\frg$ is, up to isomorphism, the Lie algebra in \eqref{eq:gslVVd}, with Lie bracket given in \eqref{eq:g[]}, for a unital Jordan algebra $(J,\cdot)$, a special unital Jordan module $T$ for $J$ (with action denoted by $a\bullet x$ for $a\in J$ and $x\in T$), and a Lie subalgebra $\frd$ of $\frg$, endowed with $\frd$-invariant bilinear maps:
\[
\begin{split}
J\times J\rightarrow \frd,&\quad (a,b)\mapsto D_{a,b}\quad\text{(skew symmetric),}\\
T\times T\rightarrow J,&\quad (x,y)\mapsto \langle x\vert y\rangle\quad\text{(skew symmetric),}\\
T\times T\rightarrow \frd,&\quad (x,y)\mapsto d_{x,y}\quad\text{(symmetric),}\\
\frd\times J\rightarrow J,&\quad (d,a)\mapsto d(a),\\
\frd\times T\rightarrow T,&\quad (d,x)\mapsto d(x),
\end{split}
\]
satisfying equations \eqref{eq:Ds}, \eqref{eq:Dabx}, \eqref{eq:Daxy}, \eqref{eq:axy}, \eqref{eq:dxya}, and \eqref{eq:dxyzdzyx}, and such that for any $d\in \frd$, the map $J\oplus T\rightarrow J\oplus T$, $a+x\mapsto d(a)+d(x)$ is a derivation of the algebra $J\oplus T$, with the product given by the formula:
\begin{equation}\label{eq:J+T}
(a+x)\diamond(b+y)=\bigl(a\cdot b+\langle x\vert y\rangle\bigr)+\bigl(a\bullet y+b\bullet x\bigr). \qed
\end{equation}
\end{theorem}

\medskip

Note that \eqref{eq:J+T} above is a consequence of the bilinear maps involved being invariant under the action of $\frd$, this being forced by the Jacobi identity on $\frg$.

\medskip

All this is strongly related to the $J$-ternary algebras considered by Allison \cite{Allison76}:

\begin{definition} (see \cite[(3.12)]{AllisonBenkartGao02}) Let $J$ be a unital Jordan algebra with multiplication $a\cdot b$. Let $T$ be a unital special Jordan module for $J$ with action $a\bullet x$ for $a\in J$ and $x\in T$. Assume $\langle.\vert.\rangle:T\times T\rightarrow J$ is a skew symmetric bilinear map and $\ptriple{.,.,.}:T\times T\times T\rightarrow T$ is a trilinear product on $T$. Then the pair $(J,T)$ is called a \emph{$J$-ternary algebra} if the following axioms hold for any $a\in J$ and $x,y,z,w,v\in T$:
\begin{description}
\addtolength{\itemindent}{-22pt}
\item[(JT1)] $a\cdot\langle x\vert y\rangle =\dfrac{1}{2}\Bigl(\langle a\bullet x\vert y\rangle +\langle x\vert a\bullet y\rangle\Bigr)$,

\item[(JT2)] $a\bullet\ptriple{ x,y,z}=\ptriple{ a\bullet x,y,z}-\ptriple{ x,a\bullet y,z} +\ptriple{x,y,a\bullet z}$,

\item[(JT3)] $\ptriple{x,y,z}=\ptriple{z,y,x}-\langle x\vert z\rangle\bullet y$,

\item[(JT4)] $\ptriple{x,y,z}=\ptriple{y,x,z}+\langle x\vert y\rangle\bullet z$,

\item[(JT5)] $\langle\ptriple{x,y,z}\vert w\rangle+\langle z\vert\ptriple{x,y,w}\rangle =\langle x\vert\langle z\vert w\rangle\bullet y\rangle$,

\item[(JT6)] $\ptriple{x,y,\ptriple{z,w,v}}=\ptriple{\ptriple{x,y,z},w,v}+
    \ptriple{z,\ptriple{y,x,w},v}+\ptriple{z,w,\ptriple{x,y,w}}$.
\end{description}
\end{definition}

\medskip

\begin{theorem}\label{th:JT}
Let $(J,\cdot)$ be a unital Jordan algebra, let $(T,\bullet)$ be a special unital Jordan module for $(J,\cdot)$, and let $\frd$ be a Lie algebra acting on both vector spaces $J$ and $T$ (that is, both $J$ and $T$ are modules for $\frd$). Assume $D_{.,.}:J\times J\rightarrow\frd$, $(a,b)\mapsto D_{a,b}$, and $\langle .\vert .\rangle:T\times T\rightarrow J$, $(x,y)\mapsto \langle x\vert y\rangle$, are two skew symmetric $\frd$-invariant bilinear maps, $d_{.,.}:T\times T\rightarrow \frd $, $(x,y)\mapsto d_{x,y}$ is a symmetric $\frd$-invariant map, and the action of $\frd$ on the direct sum $J\oplus T$ is an action by derivations of the algebra $(J\oplus T,\diamond)$ in \eqref{eq:J+T}. Then if \eqref{eq:Ds}, \eqref{eq:Dabx}, \eqref{eq:Daxy}, \eqref{eq:axy}, \eqref{eq:dxya} and \eqref{eq:dxyzdzyx} are satisfied, the pair $(J,T)$ becomes a $J$-ternary algebra with the triple product $\ptriple{.,.,.}:T\times T\times T\rightarrow T$ given by
\begin{equation}\label{eq:(xyz)ds}
\ptriple{x,y,z}=\frac{1}{2}\Bigl(-d_{x,y}(z)+\langle x\vert y\rangle\bullet z\Bigr).
\end{equation}

Conversely, if $(J,T)$ is a $J$-ternary algebra, then the image of the bilinear maps
\[
D_{.,.}:J\times J\rightarrow \End_{\bF}(J\oplus T),\ (a,b)\mapsto D_{a,b},
\]
where
\begin{equation}\label{eq:defD}
\begin{split}
D_{a,b}(c)&=a\cdot(b\cdot c)-b\cdot(a\cdot c),\\
D_{a,b}(x)&=\frac{1}{4}\Bigl(a\bullet(b\bullet x)-b\bullet (a\bullet x)\Bigr),
\end{split}
\end{equation}
for $a,b,c\in J$ and $x\in T$, and
\[
d_{.,.}:T\times T\rightarrow \End_{\bF}(J\oplus T),\ (x,y)\mapsto d_{x,y},
\]
where
\begin{equation}\label{eq:defd}
\begin{split}
d_{x,y}(a)&=\langle a\bullet x\vert y\rangle -\langle x\vert a\bullet y\rangle,\\
d_{x,y}(z)&=\langle x\vert y\rangle\bullet z-2\ptriple{x,y,z},
\end{split}
\end{equation}
for $a\in J$ and $x,y,z\in T$, is contained in the Lie algebra of derivations of the algebra $(J\oplus T,\diamond)$ defined in \eqref{eq:J+T}, and the equations \eqref{eq:Ds}, \eqref{eq:Daxy}, \eqref{eq:axy} and \eqref{eq:dxyzdzyx} are satisfied.
\end{theorem}
\begin{proof}
Assume first that $(J,\cdot)$, $(T,\bullet)$, $\langle .,.\rangle$, $D_{.,.}$ and $d_{.,.}$ satisfy the conditions in the theorem, then use \eqref{eq:(xyz)ds} to define the trilinear product $\ptriple{.,.,.}:T\times T\times T\rightarrow T$. Then condition \textbf{(JT1)} is equivalent to \eqref{eq:axy}, while condition \textbf{(JT3)} is equivalent to \eqref{eq:dxyzdzyx}. Also \textbf{(JT4)} is a consequence of the symmetry of $d_{.,.}$ and the skew symmetry of $\langle .\vert .\rangle$.

For $a\in J$ and $x,y,z\in T$ we get:
\[
\begin{split}
a\bullet&\ptriple{x,y,z} -\ptriple{x,y,a\bullet z}\\
&=\frac{1}{2}\Bigl(d_{x,y}(a\bullet z)-a\bullet d_{x,y}(z)+a\bullet\bigl(\langle x\vert y\rangle\bullet z\bigr)-\langle x\vert y\rangle\bullet (a\bullet z)\Bigr)\\
&=\frac{1}{2}\Bigl(d_{x,y}(a)\bullet z+4D_{a,\langle x\vert y\rangle}(z)\Bigr)\\
&\null\qquad\text{(using \eqref{eq:Dabx} and the fact that $d_{x,y}$ is a derivation)}\\[4pt]
&=\frac{1}{2}\Bigl(\langle a\bullet x\vert y\rangle\bullet z-\langle x\vert a\bullet y\rangle\bullet z-d_{a\bullet x,y}(z)+d_{x,a\bullet y}(z)\Bigr)\\
&\null\qquad\text{(because of \eqref{eq:Daxy} and \eqref{eq:dxya})}\\[4pt]
&=\ptriple{a\bullet x,y,z}-\ptriple{x,a\bullet y,z},
\end{split}
\]
thus obtaining the validity of \textbf{(JT2)}. Now, for $x,y,z,w\in T$ we have:
\[
\begin{split}
\langle&\ptriple{x,y,z}\vert w\rangle+\langle z\vert \ptriple{x,y,w}\rangle\\
&=\frac{1}{2}\Bigl(-\langle d_{x,y}(z)\vert w\rangle -\langle z\vert d_{x,y}(w)\rangle +\langle\langle x\vert y\rangle\bullet z\vert w\rangle +\langle z\vert\langle x\vert y\rangle\bullet w\rangle\Bigr)\\
&=\frac{1}{2}\Bigl(-d_{x,y}\bigl(\langle z\vert w\rangle\bigr)+2\langle x\vert y\rangle\cdot\langle z\vert w\rangle\Bigr)\\
&\null\qquad\text{(using \eqref{eq:axy} and the fact that $d_{x,y}$ is a derivation)}\\[4pt]
&=\frac{1}{2}\Bigl(-\langle\langle z\vert w\rangle\bullet x\vert y\rangle +\langle x\vert\langle z\vert w\rangle\bullet y\rangle \\
&\qquad\qquad +\langle\langle z\vert w\rangle\bullet x\vert y\rangle +\langle x\vert \langle z\vert w\rangle\bullet y\rangle\Bigr)
 \quad\text{(using \eqref{eq:dxya} and \eqref{eq:axy})}\\[4pt]
&=\langle x\vert\langle z\vert w\rangle\bullet y\rangle,
\end{split}
\]
thus getting \textbf{(JT5)}. Now the symmetry of $d_{.,.}$ and the fact that $\ptriple{.,.,.}$ is a $\frd$-invariant map show that \textbf{(JT6)} is a consequence of \textbf{(JT2)} and the skew symmetry of $\langle .\vert.\rangle$.

\smallskip

Conversely, let $(J,T)$ be a $J$-ternary algebra, and consider $\frd=\der(J\oplus T,\diamond)$ and $D_{.,.}$ and $d_{.,.}$ defined by equations \eqref{eq:defD} and \eqref{eq:defd}. Then both $D_{.,.}$ and $d_{.,.}$ are $\frd$-invariant maps. Let us check that $D_{a,b}$ and $d_{x,y}$ belong to $\frd$ for any $a,b\in J$ and $x,y\in T$.

Actually, $D_{a,b}$ acts as a derivation of $(J,\cdot)$ since this is a Jordan algebra. Besides, the equation $D_{a,b}(c\bullet x)=D_{a,b}(c)\bullet x+c\bullet D_{a,b}(x)$ holds for any $a,b,c\in J$ and $x\in T$ since $T$ is a special module for $J$, and the condition $D_{a,b}(\langle x\vert y\rangle) =\langle D_{a,b}(x)\vert y\rangle +\langle x\vert D_{a,b}(y)\rangle$, for $a,b\in J$ and $x,y\in T$, is a consequence of \textbf{(JT1)}. Therefore $D_{a,b}\in \der(J\oplus T,\diamond)$ for any $a,b\in J$. Also, for $x,y\in T$, $d_{x,y}$ acts as a derivation of $(J,\cdot)$ because of \textbf{(JT1)} and \eqref{eq:defd}. For $x,y,z,w\in T$, \textbf{(JT5)} is equivalent to the equation
\[
\begin{split}
\langle d_{x,y}(z)\vert w\rangle +\langle z\vert d_{x,y}(w)\rangle
 -\Bigl(\langle\langle x\vert y\rangle\bullet z\vert w\rangle &+\langle z\vert\langle x\vert y\rangle\bullet w\rangle\Bigr)\\
&\quad =-2\langle x\vert \langle z\vert w\rangle \bullet y\rangle,
\end{split}
\]
and the expression in parenthesis above is, due to \textbf{(JT1)}, $2\langle x\vert y\rangle\cdot\langle z\vert w\rangle=2\langle z\vert w\rangle\langle x\vert y\rangle$, which equals (again by \textbf{(JT1)}) $\langle \langle z\vert w\rangle\bullet x\vert y\rangle +\langle x\vert\langle z\vert w\rangle\bullet y\rangle$. Hence \textbf{(JT5)} becomes, due to \textbf{(JT1)}, the equation
\[
\langle d_{x,y}(z)\vert w\rangle+\langle z\vert d_{x,y}(w)\rangle =\langle \langle z\vert w\rangle\bullet x\vert y\rangle-\langle x\vert\langle z\vert w\rangle\bullet y\rangle =d_{x,y}\bigl(\langle z\vert w\rangle\bigr).
\]
Finally, \textbf{(JT4)} shows that $d_{x,y}(z)=-\Bigl(\ptriple{x,y,z}+\ptriple{y,x,z}\Bigr)$, which is symmetric on $x,y$. For $a\in J$ and $x,y,z\in T$ we get:
\[
\begin{split}
d_{x,y}(a&\bullet z)-d_{x,y}(a)\bullet z-a\bullet d_{x,y}(z)\\
&=-\ptriple{x,y,a\bullet z}-\ptriple{y,x,a\bullet z}-\langle a\bullet x\vert y\rangle\bullet z+\langle x\vert a\bullet y\rangle\bullet z\\
&\qquad +a\bullet \ptriple{x,y,z}+a\bullet\ptriple{y,x,z}\\
&=\Bigl(\ptriple{a\bullet x,y,z}-\ptriple{y,a\bullet x,z}-\langle a\bullet x\vert y\rangle\bullet z\Bigr)\\
&\qquad + \Bigl(\ptriple{a\bullet y,x,z}-\ptriple{x,a\bullet y,z}-\langle a\bullet y\vert x\rangle\bullet z\Bigr)\quad\text{(thanks to \textbf{(JT2)})}\\
&=0\quad\text{(because of \textbf{(JT4)} and using the skew symmetry of $\langle .\vert.\rangle$).}
\end{split}
\]
Therefore, the images of the maps $D_{.,.}$ and $d_{.,.}$ are contained in the Lie algebra of derivations of $(J\oplus T,\diamond)$.

Now, \eqref{eq:Ds} is a consequence of $J$ being a Jordan algebra, \eqref{eq:axy} is equivalent to \textbf{(JT1)}, \eqref{eq:dxyzdzyx} is equivalent to \textbf{(JT3)}. And as for the validity of \eqref{eq:Daxy}: $4D_{a,\langle x\vert y\rangle}=-d_{a\bullet x,y}+d_{x,a\bullet y}$ for any $a\in J$ and $x,y\in T$, this is valid on $J$ as a consequence of \textbf{(JT1)}, while for $z\in T$:
\[
\begin{split}
4D_{a,\langle x\vert y\rangle}&(z)+d_{a\bullet x,y}(z)-d_{x,a\bullet y}(z)\\
&=a\bullet\bigl(\langle x\vert y\rangle\bullet z\bigr)-\langle x\vert y\rangle\bullet (a\bullet z)\\
&\quad -\Bigl(\ptriple{a\bullet x,y,z}+\ptriple{y,a\bullet x,z}\Bigr)+\Bigl(\ptriple{x,a\bullet y,z}+\ptriple{a\bullet y,x,z}\Bigr)\\
&=a\bullet\Bigl(\ptriple{x,y,z}-\ptriple{y,x,z}\Bigr)-\Bigl(\ptriple{x,y,a\bullet z}-\ptriple{y,x,a\bullet z}\Bigr)\\
&\quad -\Bigl(\ptriple{a\bullet x,y,z}+\ptriple{y,a\bullet x,z}\Bigr)+\Bigl(\ptriple{x,a\bullet y,z}+\ptriple{a\bullet y,x,z}\Bigr)\\
&\qquad\qquad\qquad\text{(where we have used \textbf{(JT4)})}\\
&=0\qquad\text{(because of \textbf{(JT2)}).\qedhere}
\end{split}
\]
\end{proof}

\smallskip

\begin{corollary}\label{co:gJT}
Let $(J,T)$ be a $J$-ternary algebra and let $V$ be a two-dimensional vector space endowed with a nonzero skew symmetric bilinear map $(.\vert.)$. Then the vector space
\begin{equation}\label{eq:gJT}
\frg(J,T)=\bigl(\frsl(V)\otimes J)\oplus\bigl(V\otimes T\bigr)\oplus \frd,
\end{equation}
where $\frd$ is the linear span of the operators $D_{a,b}$ and $d_{x,y}$ in $\End_\bF(J\oplus T)$ defined in \eqref{eq:defD} and \eqref{eq:defd}, for $a,b\in J$ and $x,y\in T$, and with bracket given in \eqref{eq:g[]}, is a $BC_1$-graded Lie algebra of type $C_1$. \qed
\end{corollary}

\medskip

\section{$J$-ternary algebras and special Freudenthal-Kantor triple systems}%
\label{se:JternarySpecialFKTS}

In \cite{YO84}, Yamaguti and Ono considered a wide class of triple systems: the $(\epsilon,\delta)$ Freudenthal-Kantor triple systems, which extend the classical Freudenthal triple systems and are useful tools in the construction of Lie algebras and superalgebras.

An $(\epsilon,\delta)$ \emph{Freudental-Kantor triple system} ($\epsilon,\delta$ are either $1$ or $-1$) over a ground field $\bF$ is a triple system $(U,xyz)$ such that, if $L(x,y),K(x,y)\in\End_{\bF}(U)$ are defined by
\begin{equation}\label{eq:LK}
\left\{\begin{aligned}
L(x,y)z&=xyz\\
K(x,y)z&=xzy-\delta yzx,
\end{aligned}\right.
\end{equation}
then
\begin{subequations}\label{eq:FK}
\begin{gather}
[L(u,v),L(x,y)]=L\bigl(L(u,v)x,y\bigr)+\epsilon L\bigl(x,L(v,u)y\bigr),\label{eq:FK1}\\
K\bigl(K(u,v)x,y\bigr)=L(y,x)K(u,v)-\epsilon K(u,v)L(x,y),\label{eq:FK2}
\end{gather}
hold for any $x,y,u,v\in U$.
\end{subequations}

For $\epsilon=-1$ and $\delta=1$, these are the so called \emph{Kantor triple systems} (or generalized Jordan triple systems of second order \cite{Kan73}).

The aim of this section is to show the close relationship between the $J$-ternary algebras and some particular $(1,1)$ Freudenthal-Kantor triple systems.

\begin{lemma}\label{le:FK1}
Let $(U,xyz)$ be a triple system satisfying equation \eqref{eq:FK1}, with $\epsilon\in\{\pm1\}$, and define the endomorphisms $S(x,y)$ and $T(x,y)\in\End_\bF(U)$ by
\begin{equation}\label{eq:STs}
\begin{split}
S(x,y)&=L(x,y)+\epsilon L(y,x)\\
T(x,y)&=L(y,x)-\epsilon L(x,y).
\end{split}
\end{equation}
Then for any $u,v\in U$, $S(u,v)$ is a derivation of the triple system $(U,xyz)$, while $T(u,v)$ satisfies
\[
T(u,v)\bigl(xyz\bigr)=\bigl(T(u,v)x\bigr)yz-x\bigl(T(u,v)y\bigr)z+xy\bigl(T(u,v)z\bigr)
\]
for any $x,y,z\in U$. As a consequence, the following equations hold:
\begin{subequations}\label{subeq:STs}
\begin{align}
[S(u,v),L(x,y)]&=L\bigl(S(u,v)x,y\bigr)+L\bigl(x,S(u,v)y\bigr)\label{eq:[SL]}\\
[T(u,v),L(x,y)]&=L\bigl(T(u,v)x,y\bigr)-L\bigl(x,T(u,v)y\bigr)\label{eq:[TL]}\\
[S(u,v),T(x,y)]&=T\bigl(S(u,v)x,y\bigr)+T\bigl(x,S(u,v)y\bigr)\label{eq:[ST]}\\
[T(u,v),S(x,y)]&=-\epsilon T\bigl(T(u,v)x,y\bigr)+\epsilon T\bigl(x,T(u,v)y\bigr)\label{eq:[TS]}\\
[T(u,v),T(x,y)]&=-\epsilon S\bigl(T(u,v)x,y\bigr)+\epsilon S\bigl(x,T(u,v)y\bigr)\label{eq:[TT]}
\end{align}
\end{subequations}
\end{lemma}
\begin{proof}
The fact that $S(u,v)$ is a derivation and that $T(u,v)$ satisfies the equation above when applied to a product $xyz$ follow at once from \eqref{eq:FK1}. These conditions are equivalent to equations \eqref{eq:[SL]} and \eqref{eq:[TL]}. The other equations are obtained from these.
\end{proof}

\smallskip

\begin{remark} Let $(U,xyz)$ be a triple system satisfying equation \eqref{eq:FK1}, with $\epsilon\in\{\pm1\}$,  and denote by $\calL$ (respectively $\calS$, $\calT$) the linear span of the operators $L(x,y)$ (respectively $S(x,y)$, $T(x,y)$) for $x,y\in U$. Lemma \ref{le:FK1} shows that $\calS$ is a subalgebra of $\calL$, and that the conditions $\calL=\calS+\calT$, $[\calS,\calT]\subseteq \calT$ and $[\calT,\calT]\subseteq \calS$ hold. In particular, this shows that $\calS\cap\calT$ is an ideal of $\calL$, and modulo this ideal we get a $\bZ_2$-graded Lie algebra.
\end{remark}

\smallskip

\begin{definition}
Let $U$ be an $(\epsilon,\delta)$ Freudenthal-Kantor triple system. Then $U$ is said to be \emph{special} in case
\begin{equation}\label{eq:special}
K(x,y)=\epsilon\delta L(y,x)-\epsilon L(x,y)
\end{equation}
holds for any $x,y\in U$.

Moreover, $U$ is said to be \emph{unitary} in case the identity map belongs to $K(U,U)$ (the linear span of the endomorphisms $K(x,y)$).
\end{definition}

\smallskip

Especially, any balanced system (that is, $K(x,y)=b(x,y)1$ for some nonzero bilinear form $b(x,y)$) is unitary.

\smallskip

Therefore, an $(\epsilon,\epsilon)$ Freudenthal-Kantor triple system is special if the operators $K(x,y)$ coincide with our previous $T(x,y)$ considered in Lemma \ref{le:FK1}.

\begin{proposition}\label{pr:unitalspecial}
Let $U\ne 0$ be a unitary $(\epsilon,\delta)$ Freudenthal-Kantor triple system. Then $\epsilon=\delta$ and $U$ is special.
\end{proposition}
\begin{proof}
If $U$ is unitary, then $1=\sum_{i=1}^mK(u_i,v_i)$ for some $u_i,v_i\in U$. Replace $u$ and $v$ by $u_i$ and $v_i$ in \eqref{eq:FK2} and sum over $i$ to get $K(x,y)=L(y,x)-\epsilon L(x,y)$. For $\epsilon=\delta$ we get \eqref{eq:special}, while for $\epsilon=-\delta$ this shows $K(x,y)=-\epsilon K(y,x)$, and \eqref{eq:LK} shows that $K(x,y)=-\delta K(y,x)$, so that $K(x,y)=0$ for any $x,y\in U$, a contradiction to $U$ being unitary.
%
%
\end{proof}

\smallskip

\begin{proposition}\label{pr:Ss} \null\quad
\begin{romanenumerate}
\item
Let $U$ be a special $(\epsilon,\epsilon)$ Freudenthal-Kantor triple system. Then we have
\begin{equation}\label{eq:KKs}
\begin{split}
K(u,v)K(x,y)&+K(x,y)K(u,v)\\
&=K\bigl(K(u,v)x,y\bigr)+K\bigl(x,K(u,v)y\bigr),
\end{split}
\end{equation}
for any $u,v,x,y\in U$.

Conversely, if the triple product on $U$ satisfies equations \eqref{eq:FK1}, \eqref{eq:special} and \eqref{eq:KKs}, then it defines a special $(\epsilon,\epsilon)$ Freudenthal-Kantor triple system.

\item
Let $U$ be a special $(\epsilon,-\epsilon)$ Freudenthal-Kantor triple system. Then $K(x,y)$ is a derivation of $U$ for any $x,y\in U$, and we have
\begin{subequations}
\begin{gather}
xyz+\epsilon xzy=0,\label{eq:eesymmetry}\\
K(u,v)T(x,y)+T(x,y)K(u,v)=K\bigl(K(u,v)x,y\bigr)-K\bigl(x,K(u,v)y\bigr),\label{eq:KKT}
\end{gather}
\end{subequations}
for any $u,v,x,y\in U$.

Conversely, if $\epsilon=-\delta$ and the triple product on $U$ satisfies \eqref{eq:FK1}, \eqref{eq:eesymmetry} and \eqref{eq:KKT}, then it defines a special $(\epsilon,-\epsilon)$ Freudenthal-Kantor triple system.
\end{romanenumerate}
\end{proposition}
\begin{proof} For (i), using \eqref{eq:FK2} we compute:
\[
\begin{split}
K\bigl(&K(u,v)x,y\bigr)+K\bigl(x,K(u,v)y\bigr)=K\bigl(K(u,v)x,y\bigr)-\epsilon K\bigl(K(u,v)y,x\bigr)\\
&=\Bigl(L(y,x)-\epsilon L(x,y)\Bigr)K(u,v)+K(u,v)\Bigl(-\epsilon L(x,y)+L(y,x)\Bigr)\\
&=K(x,y)K(u,v)+K(u,v)K(x,y),
\end{split}
\]
which proves \eqref{eq:KKs}.

For the converse statement, we must check the validity of \eqref{eq:FK2}. Note that $L(x,y)=\frac{1}{2}\bigl(S(x,y)-\epsilon K(x,y)\bigr)$, and $K(x,y)$ is our previous $T(x,y)$ by \eqref{eq:special}.  Hence we compute:
\[
\begin{split}
L(y,&x)K(u,v)-\epsilon K(u,v)L(x,y)\\
&=\frac{1}{2}\Bigl(S(y,x)-\epsilon K(y,x)\Bigr)K(u,v)-\frac{1}{2}\epsilon K(u,v)\Bigl(S(x,y)-\epsilon K(x,y)\Bigr)\\
&=\frac{1}{2}\epsilon [S(x,y),K(u,v)]+\frac{1}{2}\Bigl(K(x,y)K(u,v)+K(u,v)K(x,y)\Bigr)\\
&=-\frac{1}{2}\epsilon[K(u,v),S(x,y)]+\frac{1}{2}\Bigl(K(x,y)K(u,v)+K(u,v)K(x,y)\Bigr)\\
&=K\bigl(K(u,v)x,y\bigr)\quad\text{because of \eqref{eq:[TS]} and \eqref{eq:KKs}.}
\end{split}
\]

Now for (ii), assume $U$ is a special $(\epsilon,-\epsilon)$ Freudenthal-Kantor triple system. Then $K(x,y)=-\epsilon S(x,y)$ for any $x,y$ and hence $K(x,y)$ is a derivation of $U$ (Lemma \ref{le:FK1}). Then using \eqref{eq:FK2} with $\delta=-\epsilon$ we get:
\[
\begin{split}
K\bigl(K(u,v)x&,y\bigr)-K\bigl(x,K(u,v)y\bigr)\\
&=K\bigl(K(u,v)x,y\bigr)-\epsilon K\bigl(K(u,v)y,x\bigr)\\
&=\Bigl(L(y,x)K(u,v)-\epsilon K(u,v)L(x,y)\Bigr)\\
&\qquad -\epsilon\Bigl(L(x,y)K(u,v)-\epsilon K(u,v)L(y,x)\Bigr)\\
&=T(x,y)K(u,v)+K(u,v)T(x,y)
\end{split}
\]
for any $u,v,x,y\in U$. Consider too the trilinear map $\Lambda(x,y,z)=xyz+\epsilon xzy$. Equations \eqref{eq:LK} and \eqref{eq:special} give:
\[
xzy-\delta yzx=\epsilon\delta yxz-\epsilon xyz,
\]
which shows $\Lambda(x,y,z)=\delta\Lambda(y,x,z)$ for any $x,y,z\in U$. But by its own definition $\Lambda(x,y,z)=\epsilon\Lambda(x,z,y)$. Hence with $\epsilon=-\delta$ this gives $\Lambda(x,z,y)=-\Lambda(y,x,z)$, and hence
\[
\Lambda(x,z,y)=-\Lambda(y,x,z)=\Lambda(z,y,x)=-\Lambda(x,z,y),
\]
and we get $\Lambda(x,y,z)=0$ for any $x,y,z\in U$. (Note that the argument above shows that for special $(\epsilon,\epsilon)$ Freudenthal-Kantor triple systems, $\Lambda(x,y,z)$ is symmetric on its arguments for $\epsilon=1$ and alternating for $\epsilon=-1$.)

Conversely, with $\epsilon=-\delta$, by \eqref{eq:eesymmetry} $K(x,y)=-L(y,x)-\epsilon L(x,y)=-\epsilon S(x,y)$ is a derivation of $U$, and $T(x,y)=L(y,x)-\epsilon L(x,y)$, so $L(x,y)=\frac{-\epsilon}{2}\bigl(T(x,y)+K(x,y)\bigr)$ and $L(y,x)=\frac{1}{2}\bigl(T(x,y)-K(x,y)\bigr)$. Therefore, using that $K(u,v)$ is a derivation of $U$ and \eqref{eq:KKT} we obtain
\[
\begin{split}
L(y&,x)K(u,v)-\epsilon K(u,v)L(x,y)\\
&=\frac{1}{2}\bigl(T(x,y)-K(x,y)\bigr)K(u,v)+\frac{1}{2}K(u,v)\bigl(T(x,y)+K(x,y)\bigr)\\
&=\frac{1}{2}\bigl(K(u,v)T(x,y)+T(x,y)K(u,v)\bigr)+\frac{1}{2}[K(u,v),K(x,y)]\\
&=\frac{1}{2}\Bigl(K\bigl(K(u,v)x,y\bigr)-K\bigl(x,K(u,v)y\bigr)\Bigr)\\
&\qquad\qquad +
\frac{1}{2}\Bigl(K\bigl(K(u,v)x,y\bigr)+K\bigl(x,K(u,v)y\bigr)\Bigr)\\
&=K\bigl(K(u,v)x,y\bigr)
\end{split}
\]
for any $u,v,x,y\in U$, thus getting \eqref{eq:FK2}.
\end{proof}

\smallskip

\begin{remark}\label{eq:UspecialJspecial}
Given an $(\epsilon,\epsilon)$ Freudenthal-Kantor triple system, equation \eqref{eq:KKs} shows that the linear subspace $\bF 1+\espan{K(x,y): x,y\in U}$ is a Jordan subalgebra of the special Jordan algebra $\End_\bF(U)^+$ (with product $f\cdot g=\frac{1}{2}(fg+gf)$).
\end{remark}

\medskip

\begin{theorem}\label{th:JT11}
Let $(J,T)$ be a $J$-ternary algebra. Then $T$, endowed with its triple product $\ptriple{x,y,z}$, is a special $(1,1)$ Freudenthal-Kantor triple system.

Conversely, let $(U,xyz)$ be a special $(1,1)$ Freudenthal-Kantor triple system. Let $J$ be the $\bF$ subspace of $\End_{\bF}(U)$ spanned by the identity map and by the operators $K(x,y)$ for $x,y\in U$: $J=\bF 1+K(U,U)$. Then $J$ is a subalgebra of the special Jordan algebra $\End_\bF(U)^+$ (with multiplication $f\cdot g=\frac{1}{2}(fg+gf)$, and the pair $(J,U)$, endowed with the natural action of $J$ on $U$: $a\bullet x=a(x)$, and the multilinear maps:
\[
\begin{split}
\ptriple{.,.,.}:U\times U\times U&\rightarrow U,\quad \ptriple{x,y,z}=xyz,\\
\langle .\vert.\rangle: U\times U&\rightarrow J,\quad \langle x\vert y\rangle =-K(x,y),
\end{split}
\]
is a $J$-ternary algebra.
\end{theorem}
\begin{proof}
Assume first that $(J,T)$ is a $J$-ternary algebra and consider the space $T$ endowed with the triple product $xyz=\ptriple{x,y,z}$. Take $\epsilon=\delta=1$. Then \eqref{eq:FK1} is equivalent to \textbf{(JT6)}, while $K(x,y)z=\ptriple{x,z,y}-\ptriple{y,z,x}=-\langle x\vert y\rangle\bullet z$ by \textbf{(JT3)}. Then \textbf{(JT4)} shows that $K(x,y)z=-\langle x\vert y\rangle\bullet z=\ptriple{y,x,z}-\ptriple{x,y,z}=\bigl(L(y,x)-L(x,y)\bigr)z$, which gives the speciality condition \eqref{eq:special}. Finally, \textbf{(JT5)} gives
\[
-K\bigl(L(x,y)z,w\bigr)-K\bigl(z,L(x,y)w\bigr)=K\bigl(x,K(z,w)y\bigr),
\]
By skew symmetry we get
\begin{equation}\label{eq:KKKLKL}
K\bigl(K(z,w)y,x\bigr)=K\bigl(L(x,y)z,w\bigr)+K\bigl(z,L(x,y)w\bigr).
\end{equation}
But \eqref{eq:FK1} gives, for any $x,y,z,w,u$:
\[
\begin{split}
L(x,y)K(z,w)u&=L(x,y)\bigl(zuw-wuz)\\
&=K\bigl(L(x,y)z,w\bigr)u+K\bigl(z,L(x,y)w\bigr)u+K(z,w)L(y,x)u,
\end{split}
\]
so that
\[
K\bigl(L(x,y)z,w\bigr)+K\bigl(z,L(x,y)w\bigr)=L(x,y)K(z,w)-K(z,w)L(y,x),
\]
and \eqref{eq:KKKLKL} becomes \eqref{eq:FK2}.

Conversely, let $(U,xyz)$ be a special $(1,1)$ Freudenthal-Kantor triple system, and let $J=\bF 1+K(U,U)$. Then \eqref{eq:KKs} shows that $J$ is a Jordan subalgebra of $\End_{\bF}(U)^+$, and hence $U$ becomes a unital special Jordan module for $J$. Now consider the trilinear product $\ptriple{x,y,z}=xyz$ and the skew symmetric bilinear map $\langle .\vert.\rangle:U\times  U\rightarrow J$ given by $\langle x\vert y\rangle=-K(x,y)$.

Then equation \eqref{eq:KKs} is equivalent to the condition \textbf{(JT1)}, equation \eqref{eq:FK1} and \eqref{eq:special} give \textbf{(JT2)}, equation \textbf{(JT3)} is  a consequence of the definition of the operator $K(x,y)$ in \eqref{eq:LK}, while \textbf{(JT4)} is equivalent to \eqref{eq:special}. Now \textbf{(JT5)} follows from \eqref{eq:FK2} and the skew symmetry of $K(.,.)$, and \textbf{(JT6)} is equivalent to \eqref{eq:FK1}. Therefore, $(J,U)$ is a $J$-ternary algebra.
\end{proof}

\begin{corollary}\label{co:11Lie}
Let $\frg$ be a $BC_1$-graded Lie algebra of type $C_1$ as in \eqref{eq:gslVVd}. Then the vector space $T$, endowed with the triple product \[
\ptriple{x,y,z}=\frac{1}{2}\bigl(-d_{x,y}(z)+\langle x\vert y\rangle\bullet z\bigr)
\]
in \eqref{eq:(xyz)ds}, is a special $(1,1)$ Freudenthal-Kantor triple system.

Conversely, given a special $(1,1)$ Freudenthal-Kantor triple system $(U,xyz)$, let $\frd$ be the linear span of the derivations $S(x,y)=L(x,y)+L(y,x)$ in \eqref{eq:STs} (which is a subalgebra of the Lie algebra of derivations of $\der U$ by Lemma \ref{le:FK1}) and let $J$ be the Jordan subalgebra $\bF 1+K(U,U)$ of $\End_{\bF}(U)^+$. Then the vector space
\[
\frg=\bigl(\frsl(V)\otimes J\bigr)\oplus\bigl(V\otimes U\bigr)\oplus\frd,
\]
with Lie bracket given by the bracket in $\frd$ and, for any $f,g\in\frsl(V)$, $u,v\in V$, $a,b\in J$, $x,y\in U$ and $\varphi\in \frd$, by the formulae:
\begin{itemize}
\item $[f\otimes a,g\otimes b]=[f,g]\otimes \frac{1}{2}(ab+ba)\, +\, \frac{1}{2}\trace(fg)[a,b]$,
\item $[f\otimes a,u\otimes x]=f(u)\otimes a(x)$,
\item $[u\otimes x,v\otimes y]=-\gamma_{u,v}\otimes K(x,y)\,-\, (u\vert v)\bigl(L(x,y)+L(y,x)\bigr)$,
\item $[\varphi,f\otimes a]=f\otimes \varphi(a)$,
\item $[\varphi,u\otimes x]=u\otimes \varphi(x)$,
\end{itemize}
is a $BC_1$-graded Lie algebra of type $C_1$.
\end{corollary}
\begin{proof}
It is enough to note that if $(U,xyz)$ is a special $(1,1)$ Freudenthal-Kantor triple system and $(J,U)$ is the associated $J$-ternary algebra, with $J=\bF 1+K(U,U)$, then for $x,y,z\in U$, and $a,b,c\in J$:
\[
\begin{split}
d_{x,y}(z)&=\langle x\vert y\rangle\bullet z-2\ptriple{x,y,z}\\
&=\bigl(-K(x,y)-2L(x,y)\bigr)(z)\\
&=-\bigl(L(x,y)+L(y,x)\bigr)(z)\\[4pt]
d_{x,y}(a)&=\langle a\bullet x\vert y\rangle-\langle x\vert a\bullet y\rangle\\
&=-K\bigl(a(x),y\bigr)-K\bigl(a(y),x\bigr)\\
&=-[L(x,y)+L(y,x),a]\quad\text{(because of \eqref{eq:[TS]})}\\[4pt]
D_{a,b}(x)&=\frac{1}{4}\bigl(a(b(x))-b(a(x))\bigr)=\frac{1}{4}[a,b](x),\\
D_{a,b}(c)&=a\cdot(b\cdot c)-b\cdot(a\cdot c)=\frac{1}{4}[[a,b],c].
\end{split}
\]
(Note that we have $4\Bigl(a\cdot(b\cdot c)-b\cdot (a\cdot c)\Bigr)=a(bc+cb)+(bc+cb)a-b(ac+ca)-(ac+ca)b=(ab-ba)c-c(ab-ba)=[[a,b],c]$.)
\end{proof}

\begin{remark}
One may change the skew symmetric bilinear form on $V$ to its negative and hence change the Lie bracket accordingly. In this way one has $[u\otimes x,v\otimes y]=\gamma_{u,v}\otimes K(x,y)+(u\vert v)\bigl(L(x,y)+L(y,x)\bigr)$, which is nicer than the bracket above.
\end{remark}

\begin{remark}
Let $(U,xyz)$ be a special $(1,1)$ Freudenthal-Kantor triple system and let $\frg$ be the attached Lie algebra as in Corollary \ref{co:11Lie}. Then the subspace $\bigl(\frsl(V)\otimes K(U,U)\bigr)\oplus \bigl(V\otimes U\bigr)\oplus \frd$ is an ideal of $\frg$. Hence if $\frg$ is simple we get $J=K(U,U)$, so that $(U,xyz)$ is unitary. Moreover, if $\dim J=1$, then $(U,xyz)$ is balanced.
\end{remark}

\bigskip

\section{Superalgebras}\label{se:super}

Any $BC_1$-graded Lie algebra of type $C_1$ (as in \eqref{eq:gslVVd}) is $\bZ_2$-graded, with
\begin{equation}\label{eq:g0g1}
\frg\subo=\bigl(\frsl(V)\otimes J\bigr)\oplus\frd,\qquad \frg\subuno=V\otimes T.
\end{equation}

In a similar vein we may consider Lie superalgebras $\frg=\frg\subo\oplus\frg\subuno$ such that $\frg\subo$ contains a subalgebra isomorphic to $\frsl(V)$ so that $\frg\subo$ is a direct sum of copies of the adjoint module for $\frsl(V)$ and copies of the trivial module, and $\frg\subuno$ is a direct sum of copies of the two-dimensional module $V$. These will be called \emph{strictly $BC_1$-graded Lie superalgebras of type $C_1$}.

The arguments in the two previous sections apply here. The bracket on the superalgebra $\frg$ is given by \eqref{eq:g[]} for suitable bilinear maps as in \eqref{eq:maps} with the difference that $\langle.\vert.\rangle:T\times T\rightarrow J$ is symmetric here, while $d_{.,.}:T\times T\rightarrow \frd$ is skew symmetric.

Theorem \ref{th:gBC1} has the following counterpart (whose proof is obtained by superizing the arguments there):

\begin{theorem}\label{th:gsuperBC1}
Let $\frg$ be a Lie superalgebra. Then $\frg$ is strictly $BC_1$-graded of type $C_1$ if and only if there is a two-dimensional vector space $V$ such that that $\frg$ is, up to isomorphism, the Lie superalgebra in \eqref{eq:gslVVd}, with $\frg\subo=\bigl(\frsl(V)\otimes J\bigr)\oplus\frd$ and $\frg\subuno=V\otimes T$, and Lie bracket as in \eqref{eq:g[]}, for a unital Jordan algebra $(J,\cdot)$, a special unital Jordan module $T$ for $J$ (with action denoted by $a\bullet x$ for $a\in J$ and $x\in T$), and a Lie subalgebra $\frd$ of $\frg$, endowed with $\frd$-invariant bilinear maps:
\[
\begin{split}
J\times J\rightarrow \frd,&\quad (a,b)\mapsto D_{a,b}\quad\text{(skew symmetric),}\\
T\times T\rightarrow J,&\quad (x,y)\mapsto \langle x\vert y\rangle\quad\text{(symmetric),}\\
T\times T\rightarrow \frd,&\quad (x,y)\mapsto d_{x,y}\quad\text{(skew symmetric),}\\
\frd\times J\rightarrow J,&\quad (d,a)\mapsto d(a),\\
\frd\times T\rightarrow T,&\quad (d,x)\mapsto d(x),
\end{split}
\]
satisfying the following equations for any $a,b,c\in J$ and $x,y,z\in T$:
\begin{align}
&D_{a\cdot b,c}+D_{b\cdot c,a}+D_{c\cdot a,b}=0,\ D_{a,b}(c)=a\cdot(b\cdot c)-b\cdot(a\cdot c),\label{eq:superDs}\\
&4D_{a,b}(x)=a\bullet(b\bullet x)-b\bullet(a\bullet x),\label{eq:superDabx}\\
&4D_{a,\langle x\vert y\rangle}=-d_{a\bullet x,y}+d_{x,a\bullet y},\label{eq:superDaxy}\\
&2a\cdot\langle x\vert y\rangle=\langle a\bullet x\vert y\rangle+\langle x\vert a\bullet y\rangle,\label{eq:superaxy}\\
&d_{x,y}(a)=\langle a\bullet x\vert y\rangle-\langle x\vert a\bullet y\rangle,\label{eq:superdxya}\\
&d_{x,y}(z)-d_{y,z}(x)=\langle x\vert y\rangle\bullet z+\langle y\vert z\rangle\bullet x-2\langle z\vert x\rangle\bullet y,\label{eq:superdxyzdzyx}
\end{align}
and such that for any $d\in \frd$, the map $J\oplus T\rightarrow J\oplus T$, $a+x\mapsto d(a)+d(x)$ is a derivation of the algebra $J\oplus T$, with the product given by the formula \eqref{eq:J+T}.
\end{theorem}

\smallskip

The definition of a $J$-ternary algebra is changed to:

\begin{definition}  Let $J$ be a unital Jordan algebra with multiplication $a\cdot b$. Let $T$ be a unital special Jordan module for $J$ with action $a\bullet x$ for $a\in J$ and $x\in T$. Assume $\langle.\vert.\rangle:T\times T\rightarrow J$ is a symmetric bilinear map and $\ptriple{.,.,.}:T\times T\times T\rightarrow T$ is a trilinear product on $T$. Then the pair $(J,T)$ is called a \emph{$J$-ternary $(-1)$-algebra} if the following axioms hold for any $a\in J$ and $x,y,z,w,v\in T$:
\begin{description}
\addtolength{\itemindent}{-22pt}
\item[(-JT1)] $a\cdot\langle x\vert y\rangle =\dfrac{1}{2}\Bigl(\langle a\bullet x\vert y\rangle +\langle x\vert a\bullet y\rangle\Bigr)$,

\item[(-JT2)] $a\bullet\ptriple{ x,y,z}=\ptriple{ a\bullet x,y,z}-\ptriple{ x,a\bullet y,z} +\ptriple{x,y,a\bullet z}$,

\item[(-JT3)] $\ptriple{x,y,z}+\ptriple{z,y,x}=\langle x\vert z\rangle\bullet y$,

\item[(-JT4)] $\ptriple{x,y,z}+\ptriple{y,x,z}=\langle x\vert y\rangle\bullet z$,

\item[(-JT5)] $\langle\ptriple{x,y,z}\vert w\rangle+\langle z\vert\ptriple{x,y,w}\rangle =\langle x\vert\langle z\vert w\rangle\bullet y\rangle$,

\item[(-JT6)] $\ptriple{x,y,\ptriple{z,w,v}}=\ptriple{\ptriple{x,y,z},w,v}-
    \ptriple{z,\ptriple{y,x,w},v}+\ptriple{z,w,\ptriple{x,y,w}}$.
\end{description}
\end{definition}

\medskip

Now Theorem \ref{th:JT} has the following counterpart:

\begin{theorem}\label{th:-JT}
Let $(J,\cdot)$ be a unital Jordan algebra, let $(T,\bullet)$ be a special unital Jordan module for $(J,\cdot)$, and let $\frd$ be a Lie algebra acting on both vector spaces $J$ and $T$ (that is, both $J$ and $T$ are modules for $\frd$). Assume $D_{.,.}:J\times J\rightarrow\frd$, $(a,b)\mapsto D_{a,b}$, and $d_{.,.}:T\times T\rightarrow \frd $, $(x,y)\mapsto d_{x,y}$, are two skew symmetric $\frd$-invariant bilinear maps,
$\langle .\vert .\rangle:T\times T\rightarrow J$, $(x,y)\mapsto \langle x\vert y\rangle$, is a symmetric $\frd$-invariant bilinear map, and
the action of $\frd$ on the direct sum $J\oplus T$ is an action by derivations of the algebra $(J\oplus T,\diamond)$ in \eqref{eq:J+T}. Then if \eqref{eq:superDs}, \eqref{eq:superDabx}, \eqref{eq:superDaxy}, \eqref{eq:superaxy}, \eqref{eq:superdxya} and \eqref{eq:superdxyzdzyx} are satisfied, the pair $(J,T)$ becomes a $J$-ternary $(-1)$-algebra with the triple product $\ptriple{.,.,.}:T\times T\times T\rightarrow T$ given by \eqref{eq:(xyz)ds}:
\[
\ptriple{x,y,z}=\frac{1}{2}\Bigl(-d_{x,y}(z)+\langle x\vert y\rangle\bullet z\Bigr).
\]

Conversely, if $(J,T)$ is a $J$-ternary $(-1)$-algebra, then the image of the bilinear maps
\[
D_{.,.}:J\times J\rightarrow \End_{\bF}(J\oplus T),\ (a,b)\mapsto D_{a,b},
\]
where, as in \eqref{eq:defD}:
\[
\begin{split}
D_{a,b}(c)&=a\cdot(b\cdot c)-b\cdot(a\cdot c),\\
D_{a,b}(x)&=\frac{1}{4}\Bigl(a\bullet(b\bullet x)-b\bullet (a\bullet x)\Bigr),
\end{split}
\]
for $a,b,c\in J$ and $x\in T$, and
\[
d_{.,.}:T\times T\rightarrow \End_{\bF}(J\oplus T),\ (x,y)\mapsto d_{x,y},
\]
where, as in \eqref{eq:defd}
\[
\begin{split}
d_{x,y}(a)&=\langle a\bullet x\vert y\rangle -\langle x\vert a\bullet y\rangle,\\
d_{x,y}(z)&=\langle x\vert y\rangle\bullet z-2\ptriple{x,y,z},
\end{split}
\]
for $a\in J$ and $x,y,z\in T$, is contained in the Lie algebra of derivations of the algebra $(J\oplus T,\diamond)$ defined in \eqref{eq:J+T}, and the equations \eqref{eq:superDs}, \eqref{eq:superDaxy}, \eqref{eq:superaxy} and \eqref{eq:superdxyzdzyx} are satisfied.
\end{theorem}

Finally, Theorem \ref{th:JT11} becomes:

\begin{theorem}\label{th:-JT-1-1}
Let $(J,T)$ be a $J$-ternary $(-1)$-algebra. Then $T$, endowed with its triple product $\ptriple{x,y,z}$, is a special $(-1,-1)$ Freudenthal-Kantor triple system.

Conversely, let $(U,xyz)$ be a special $(-1,-1)$ Freudenthal-Kantor triple system. Let $J$ be the $\bF$ subspace of $\End_{\bF}(U)$ spanned by the identity map and by the operators $K(x,y)$ for $x,y\in U$. Then $J$ is a subalgebra of the special Jordan algebra $\End_\bF(U)^+$ (with multiplication $f\cdot g=\frac{1}{2}(fg+gf)$), and the pair $(J,U)$, endowed with the natural action of $J$ on $U$: $a\bullet x=a(x)$, and the multilinear maps:
\[
\begin{split}
\ptriple{.,.,.}:U\times U\times U&\rightarrow U,\quad \ptriple{x,y,z}=xyz,\\
\langle .\vert.\rangle: U\times U&\rightarrow J,\quad \langle x\vert y\rangle =K(x,y),
\end{split}
\]
is a $J$-ternary $(-1)$-algebra.
\end{theorem}

\begin{remark}
Note that $K(x,y)=-\langle x\vert y\rangle$ in the $(1,1)$-setting, but $K(x,y)=\langle x\vert y\rangle$ in the $(-1,-1)$-setting.
\end{remark}

\begin{corollary}\label{co:-1-1Lie}
Let $\frg$ be a strictly $BC_1$-graded Lie superalgebra of type $C_1$ as in \eqref{eq:gslVVd}. Then the vector space $T$, endowed with the triple product $\ptriple{x,y,z}=\frac{1}{2}\bigl(-d_{x,y}(z)+\langle x\vert y\rangle\bullet z\bigr)$ in \eqref{eq:(xyz)ds}, is a special $(-1,-1)$ Freudenthal-Kantor triple system.

Conversely, given a special $(-1,-1)$ Freudenthal-Kantor triple system $T$, let $\frd$ be the linear span of the derivations $S(x,y)=L(x,y)-L(y,x)$  and let $J$ be the Jordan subalgebra $\bF 1+K(T,T)$ of $\End_{\bF}(T)^+$. Then the vector space
\[
\frg=\bigl(\frsl(V)\otimes J\bigr)\oplus\bigl(V\otimes T\bigr)\oplus\frd,
\]
with Lie bracket given by the bracket in $\frd$ and, for any $f,g\in\frsl(V)$, $u,v\in V$, $a,b\in J$, $x,y\in T$ and $\varphi\in \frd$, by the formulae:
\begin{itemize}
\item $[f\otimes a,g\otimes b]=[f,g]\otimes \frac{1}{2}(ab+ba)\, +\, \frac{1}{2}\trace(fg)[a,b]$,
\item $[f\otimes a,u\otimes x]=f(u)\otimes a\bullet x$,
\item $[u\otimes x,v\otimes y]=\gamma_{u,v}\otimes K(x,y)\,-\, (u\vert v)\bigl(L(x,y)-L(y,x)\bigr)$,
\item $[\varphi,f\otimes a]=f\otimes\varphi(a)$,
\item $[\varphi,u\otimes x]=u\otimes \varphi(x)$,
\end{itemize}
is a strictly $BC_1$-graded Lie superalgebra of type $C_1$, with $\frg\subo=\bigl(\frsl(V)\otimes J\bigr)\oplus\frd$ and $\frg\subuno=V\otimes T$.
\end{corollary}

\medskip

\section{Dicyclic symmetry}\label{se:dicyclic}

Let $V$ be a two-dimensional vector space endowed with a nonzero (and hence nondegenerate) skew symmetric bilinear form $(.|.)$. Let $\{u,v\}$ be a symplectic basis of $V$ (that is, $(u\vert v)=1$). Assume that the ground field contains the primitive cubic roots $\omega$ and $\omega^2$ of $1$. Then the corresponding symplectic group $Sp(V)=SL(V)$ contains the elements $\theta$, $\phi$, whose coordinate matrices on the basis above are:
\begin{equation}\label{eq:thetaphi}
\begin{split}
\theta&\leftrightarrow \begin{pmatrix}0&1\\-1&0\end{pmatrix},\quad\text{(so $\theta(u)=-v$, $\theta(v)=u$),}\\
\phi&\leftrightarrow \begin{pmatrix} \omega&0\\ 0&\omega^2\end{pmatrix}.
\end{split}
\end{equation}
These elements satisfy
\begin{equation}\label{eq:theta4phi3}
\theta^4=1=\phi^3,\quad \phi\theta\phi=\theta,
\end{equation}
and hence the group they generate is the \emph{dicyclic group} $\Dic_3$ (semidirect product of $C_3$ by $C_4$).

\smallskip

It turns out that conjugation by elements of  $\Dic_3$ gives an action of this group by automorphisms on $\frsl(V)$ ($\sigma(f)=\sigma f\sigma^{-1}$ for any $\sigma\in \Dic_3$ and $f\in \frsl(V)$), and hence that any $BC_1$ graded Lie algebra of type $C_1$ (as in \eqref{eq:gslVVd}) is endowed with an action of $\Dic_3$ by automorphisms.

\smallskip

The aim of this section is the study of those Lie algebras endowed with an action of $\Dic_3$ by automorphisms.

\smallskip

Let $\frg$ be a Lie algebra endowed with such an action. Then $\phi$ induces a grading over $\bZ_3$:
\begin{equation}\label{eq:g1gomegagomega2}
\frg=\frg_1\oplus\frg_\omega\oplus\frg_{\omega^2},
\end{equation}
where $\frg_1=\{x\in \frg:\phi(x)=x\}$, $\frg_\omega=\{x\in \frg:\phi(x)=\omega x\}$ and $\frg_{\omega^2}=\{x\in\frg: \phi(x)=\omega^2x\}$.

Since $\phi\theta\phi=\theta$, the following conditions hold:
\[
\theta(\frg_1)=\frg_1,\quad \theta(\frg_\omega)=\frg_{\omega^2},\quad\theta(\frg_{\omega^2})=\frg_\omega.
\]
This allows us to define a unary, binary and ternary operations on $\frg_\omega$ as follows:
\begin{subequations}\label{eq:barsquaretriple}
\begin{gather}
\bar x=\theta^2(x),\\
x*y=\theta^{-1}\bigl([x,y]\bigr),\\
\{x,y,z\}=[[x,\theta(y)],z],
\end{gather}
\end{subequations}
for $x,y,z\in \frg_\omega$. It is easy to check that these are indeed well defined. For instance, for $x,y,z\in \frg_\omega$, $\theta(y)$ belongs to $\frg_{\omega^2}$, and hence we have $[x,\theta(y)]\in[\frg_\omega,\frg_{\omega^2}]\subseteq \frg_1$ and $\{x,y,z\}=[[x,\theta(y)],z]\in [\frg_1,\frg_\omega]\subseteq \frg_\omega$.

\begin{theorem}\label{th:dicyclic}
Let $\frg$ be a Lie algebra over a field $\bF$ containing the primitive cubic roots of $1$, endowed with an action by automorphisms of the dicyclic group $\Dic_3$. Let $\bar x$, $x*y$ and $\{x,y,z\}$ be defined as in \eqref{eq:barsquaretriple} on $\frg_\omega=\{x\in \frg:\phi(x)=\omega x\}$. Then the binary product $x*y$ is anticommutative and the map $x\mapsto \bar x$ is an order $2$ automorphism of both the binary and the ternary products: $\overline{x*y}=\bar x*\bar y$, $\overline{\{x,y,z\}}=\{\bar x,\bar y,\bar z\}$. Moreover, the following identities hold for any $u,v,,w,x,y,z\in\frg_\omega$:
\begin{description}
\addtolength{\itemindent}{-22pt}
\item[(D1)] $\{x,z,y\}-\{y,z,x\}=(\bar x*\bar y)*\bar z$,
\smallskip
\item[(D2)] $\{u,\bar v,x*y\}+\{v,u,x\}*y+x*\{v,u,y\}=0$,
\smallskip
\item[(D3)] $\{x,y*z,w\}+\{y,z*x,w\}+\{z,x*y,w\}=0$,
\smallskip
\item[(D4)] $\{\bar x*\bar y,z,w\}+\{\bar y*\bar z,x,w\}+\{\bar z*\bar x,y,w\}=0$,
\smallskip
\item[(D5)] $\{u,v,\{x,y,z\}\}=\{\{u,v,x\},y,z\}-\{x,\{v,\bar u,y\},z\}+\{x,y,\{u,v,z\}\}$.
\end{description}
\end{theorem}
\begin{proof}
The anticommutativity of $x*y$ is clear. Also, since $\theta^2$ generates the center of the dicyclic group $\Dic_3$, it follows easily that $x\mapsto \bar x=\theta^2(x)$ is an order $2$ automorphism of both the binary and ternary products.

Now for $x,y,z\in \frg_\omega$:
\[
\begin{split}
\{x,z,y\}-\{y,z,x\}&=[[x,\theta(z)],y]-[[y,\theta(z)],x]\\
  &=[[x,y],\theta(z)]=\theta\bigl([\theta^{-1}([x,y]),z]\bigr)\\
  &=\theta\bigl([x*y,z])=\theta^2\bigl(\theta^{-1}([x*y,z])\bigr)\\
  &=\overline{(x*y)*z}=(\bar x*\bar y)*\bar z,
\end{split}
\]
which gives \textbf{(D1)}.

For any $d\in\frg_1$ we have
\[
\begin{split}
[d,x*y]&=[d,\theta^{-1}([x,y])]=\theta^{-1}\bigl([\theta(d),[x,y]]\bigr)\\
  &=\theta^{-1}\bigl([[\theta(d),x],y]+[x,[\theta(d),y]]\bigr)\\
  &=[\theta(d),x]*y+x*[\theta(d),y],
\end{split}
\]
and thus
\[
\begin{split}
\{u,\bar v,x*y\}&=[[u,\theta(\bar v)],x*y]\\
  &=[\theta([u,\theta(\bar v)]),x]*y+x*[\theta([u,\theta(\bar v)]),y]\\
  &=-[[v,\theta(u)],x]*y-x*[[v,\theta(u)],y]\\
  &=-\{v,u,x\}*y-x*\{v,u,y\},
\end{split}
\]
which gives \textbf{(D2)}.

Equation \textbf{(D3)} is a direct consequence of the Jacobi identity, since $\{x,y*z,w\}=[[x,\theta(y*z)],w]=[[x,[y,z]],w]$. The same happens for \textbf{(D4)}, as $\{\bar x*\bar y,z,w\}=[[\bar x*\bar y,\theta(z)],w]=[[\overline{x*y},\theta(z)],w] =[[\theta^2\bigl(\theta^{-1}([x,y])\bigr),\theta(z)],w] =[[[\theta(x),\theta(y)],\theta(z)],w]$.

Finally,
\[
\begin{split}
\{u&,v,\{x,y,z\}\}\\
    &=[[u,\theta(v)],[[x,\theta(y)],z]]\\
  &=[[[[u,\theta(v)],x],\theta(y)],z]+[[x,[[u,\theta(v)],\theta(y)]],z]+
  [[x,\theta(y)],[[u,\theta(v)],z]]\\
  &=\{\{u,v,x\},y,z\}-[[x,\theta\bigl([[v,\theta^{-1}(u)],y]\bigr)],z]+\{x,y,\{u,v,z\}\}\\
  &=\{\{u,v,x\},y,z\}-\{x,\{v,\bar u,y\},z\}+\{x,y,\{u,v,z\}\},
\end{split}
\]
thus obtaining \textbf{(D5)}.
\end{proof}

\medskip

\begin{definition}\label{de:dicyclic}
Let $(A,\bar{\ },*,\{...\})$ be a vector space endowed with a linear endomorphism $x\mapsto \bar x$, a binary product $x*y$ and a triple product $\{x,y,z\}$. Then $(A,\bar{\ },*,\{...\})$ is said to be a \emph{dicyclic ternary algebra} (or simply a \emph{$D$-ternary algebra}) if the linear map $\bar{\ }$ is an order $2$ automorphism of both the binary and the triple products, the binary product is anticommutative, and conditions \textbf{(D1)--(D5)} are satisfied.
\end{definition}

\smallskip

Theorem \ref{th:dicyclic} has a natural converse:

\begin{theorem}\label{th:conversedicyclic}
Let $(A,\bar{\ },*,\{...\})$ be a dicyclic algebra. Let $\iota_1(A)$ and $\iota_2(A)$ denote two copies of $A$, and for any $u,v\in A$ denote by $\tau(u,v)\in\End_\bF(A)$ the linear map $x\mapsto \{u,v,x\}$, and by $\tau(A,A)$ the linear span of these operators. Then the vector space direct sum
\[
\frg(A)=\tau(A,A)\oplus\iota_1(A)\oplus\iota_2(A)
\]
is a Lie algebra under the bracket:
\[
\begin{split}
[\tau(u,v),\tau(x,y)]&=\tau\bigl(\{u,v,x\},y\bigr)-\tau\bigl(x,\{v,\bar u,y\}\bigr),\\
[\tau(u,v),\iota_1(x)]&=\iota_1\bigl(\{u,v,x\}\bigr),\\
[\tau(u,v),\iota_2(x)]&=-\iota_2\bigl(\{v,\bar u,x\}\bigr),\\
[\iota_1(x),\iota_1(y)]&=\iota_2(x*y),\\
[\iota_2(x),\iota_2(y)]&=\iota_1(\overline{x*y}),\\
[\iota_1(x),\iota_2(y)]&=\tau(x,y),
\end{split}
\]
for any $u,v,x,y\in A$.

Moreover, if the ground field $\bF$ contains the primitive cubic roots $\omega$ and $\omega^2$ of $1$, then $\frg$ is endowed with an action by automorphisms of the dicyclic group $\Dic_3$ as follows:
\[
\begin{aligned}
\phi\bigl(\tau(x,y)\bigr)&=\tau(x,y)&\quad \phi\bigl(\iota_1(x)\bigr)&=\omega\iota_1(x)&\quad \phi\bigl(\iota_2(x)\bigr)&=\omega^2\iota_2(x)\\
\theta\bigl(\tau(x,y)\bigr)&=-\tau(\bar y,x)&\theta\bigl(\iota_1(x)\bigr)&=\iota_2(x)&
\theta\bigl(\iota_2(x)\bigr)&=\iota_1(\bar x)
\end{aligned}
\]
for any $x,y\in A$. \qed
\end{theorem}

The proof consists of straightforward computations using \textbf{(D1)--(D5)}. Note that the bracket $[\tau(u,v),\tau(x,y)]$ is just, due to \textbf{(D5)}, the bracket of these operators in $\End_\bF(A)$.

\smallskip

In ending this section, we simply note that the generalized Malcev algebra based upon the $S_3$ symmetry considered in \cite{EOS3S4} is a special case of the $D$-ternary algebra in which we have $\bar x=x$ for any $x\in U$.

\medskip

\section{Dicyclic and $J$-ternary algebras}\label{se:dicyclicJternary}

Let $(J,T)$ be a $J$-ternary algebra and consider the attached Lie algebra in \eqref{eq:gJT}:
\[
\frg(J,T)=\bigl(\frsl(V)\otimes J)\oplus\bigl(V\otimes T\bigr)\oplus \frd.
\]
Then, if the ground field $\bF$ contains the primitive cubic roots $\omega,\omega^2$ of $1$, the dicyclic group $\Dic_3$ (which is a subgroup of the symplectic group $Sp(V)=SL(V)$) acts on $\frg(J,T)$ by automorphisms. Recall that the action of $\Dic_3$ on $\frsl(V)$ is by conjugation. The action on $\frd$ is trivial.

Let $\{u,v\}$ be a symplectic basis of $V$: $(u\vert v)=1$, and consider the standard basis $\{H,E,F\}$ of $\frsl(V)$, with coordinate matrices in this basis:
\begin{equation}\label{eq:matricesHEF}
H\leftrightarrow\begin{pmatrix} 1&0\\ 0&-1\end{pmatrix},\quad
E\leftrightarrow\begin{pmatrix} 0&1\\ 0&0\end{pmatrix},\quad
F\leftrightarrow\begin{pmatrix} 0&0\\ 1&0\end{pmatrix},
\end{equation}
which satisfy
\[
[H,E]=2E,\quad [H,F]=-2F,\quad [E,F]=H.
\]
Let $\gamma_{w_1,w_2}=(w_1\vert.)w_2+(w_2\vert .)w_1$ be as in Lemma \ref{le:invariantV}. Then, relative to our symplectic basis we have:
\begin{equation}\label{eq:gammaHEF}
\gamma_{u,v}=-H,\qquad\gamma_{u,u}=2E,\qquad \gamma_{v,v}=-2F.
\end{equation}
The action of the elements $\theta,\phi\in\Dic_3$ in \eqref{eq:thetaphi} induces the action by conjugation on $\frsl(V)$, which satisfies:
\[
\begin{aligned}
\phi H\phi^{-1}&=H,&\qquad \phi E\phi^{-1}&=\omega^2 E,&\qquad
\phi F\phi^{-1}&=\omega F,\\
\theta H\theta^{-1}&=-H,&\qquad \theta E\theta^{-1}&=-F,&\qquad
\theta F\theta^{-1}&=-E.
\end{aligned}
\]
(Note that $\sigma\gamma_{w_1,w_2}\sigma^{-1}=\gamma_{\sigma(w_1),\sigma(w_2)}$ for any $\sigma\in Sp(V)$ and $w_1,w_2\in V$.)

Then the Lie algebra $\frg=\frg(J,T)$ decomposes under the action of $\phi$ as in \eqref{eq:g1gomegagomega2} with
\[
\begin{split}
\frg_1&=\bigl(H\otimes J\bigr)\oplus \frd,\\
\frg_\omega&=\bigl(F\otimes J\bigr)\oplus\bigl(u\otimes T\bigr),\\
\frg_{\omega^2}&=\bigl(E\otimes J\bigr)\oplus\bigl(v\otimes T\bigr).
\end{split}
\]
Therefore $\frg_\omega=\bigl(F\otimes J\bigr)\oplus\bigl(u\otimes T\bigr)$ becomes a dicyclic ternary algebra with order $2$ automorphism given by
\[
\overline{F\otimes a}=F\otimes a,\quad \overline{u\otimes x}=-u\otimes x,
\]
for any $a\in J$ and $x\in T$ (recall from \eqref{eq:barsquaretriple} that $\bar{\ }$ is the restriction to $\frg_\omega$ of $\theta^2$),  with binary and triple products given by the following formulas (where we use \eqref{eq:g[]}):
\[
\begin{split}
(F\otimes a)*(F\otimes b)&=\theta^{-1}\bigl([F\otimes a,F\otimes b]\bigr)=0,\\[4pt]
(F\otimes a)*(u\otimes x)&=\theta^{-1}\bigl([F\otimes a,u\otimes x]\bigr)\\
    &=
     \theta^{-1}\bigl(v\otimes a\bullet x\bigr)\\
     &=-u\otimes a\bullet x\\[4pt]
(u\otimes x)*(u\otimes y)&=\theta^{-1}\bigl([u\otimes x,u\otimes y]\bigr)
      \\
      &=\theta^{-1}\bigl(2E\otimes \langle x\vert y\rangle\bigr)\\
      &=-2F\otimes\langle x\vert y\rangle,
\end{split}
\]
and with triple product given by:
\[
\begin{split}
\{F\otimes a,u\otimes x,X\}&=[[F\otimes a,\theta(u\otimes x)],X]\\
    &=-[[F\otimes a,v\otimes x],X]=0,\\
\{u\otimes x,F\otimes a,X\}&=[[u\otimes x,\theta(F\otimes a)],X]\\
    &=-[[u\otimes x,E\otimes a],X]=0,
\end{split}
\]
for any $a\in J$, $x\in T$, and by:
\[
\begin{split}
\{F\otimes a,F\otimes b,F\otimes c\}
    &=[[F\otimes a,\theta(F\otimes b)],F\otimes c]\\
    &=-[[F\otimes a,E\otimes b],F\otimes c]\\
    &=[H\otimes a\cdot b-2D_{a,b},F\otimes c]\quad\text{(as $\trace(FE)=1$)}\\
    &=-2F\otimes\bigl((a\cdot b)\cdot c+D_{a,b}(c)\bigr)\\
    &=-2F\otimes\bigl( (a\cdot b)\cdot c+a\cdot(b\cdot c)-(a\cdot c)\cdot b\bigr),\\[4pt]
\{F\otimes a,F\otimes b,u\otimes x\}
    &=[H\otimes a\cdot b-2D_{a,b},u\otimes x]\\
    &=u\otimes\bigl((a\cdot b)\bullet x-2D_{a,b}(x)\bigr)\\
    &=u\otimes (b\bullet(a\bullet x)),\quad\text{(using \eqref{eq:Dabx} and \eqref{eq:abx})}\\[4pt]
\{u\otimes x,u\otimes y,F\otimes a\}
    &=[[u\otimes x,\theta(u\otimes y)],F\otimes a]\\
    &=-[[u\otimes x,v\otimes y],F\otimes a]\\
    &=[H\otimes\langle x\vert y\rangle-d_{x,y},F\otimes a]\\
    &=F\otimes\bigl(-2\langle x\vert y\rangle\cdot a-d_{x,y}(a)\bigr)\\
    &=-2F\otimes \langle a\bullet x\vert y\rangle,\quad\text{(using \eqref{eq:axy} and \eqref{eq:dxya})}\\[4pt]
\{u\otimes x,u\otimes y,u\otimes z\}
    &=[H\otimes \langle x\vert y\rangle-d_{x,y},u\otimes z]\\
    &=u\otimes\bigl(\langle x\vert y\rangle\bullet z-d_{x,y}(z)\bigr)\\
    &=2u\otimes \ptriple{x,y,z},\quad\text{(using \eqref{eq:(xyz)ds})}
\end{split}
\]
for any $a,b\in J$, $x,y,z\in T$ and $X\in \frg_\omega$.

\smallskip

We may identify $\frg_\omega=\bigl(F\otimes J\bigr)\oplus\bigl(u\otimes T\bigr)$ with $A=J\oplus T$, by means of $(F\otimes a)+(u\otimes x)\leftrightarrow a+x\in J\oplus T$.

\smallskip

The above computations are summarized in the next result.

\begin{proposition}\label{pr:JTdicyclic}
Let $(J,T)$ be a $J$-ternary algebra over a field $\bF$ containing the primitive cubic roots $\omega,\omega^2$ of $1$, and let $\frg=\frg(J,T)$ be the attached $BC_1$-graded Lie algebra of type $C_1$ with the above action of the dicyclic group $\Dic_3$ by automorphisms. Then the dicyclic ternary algebra defined on $\frg_\omega$ by means of \eqref{eq:barsquaretriple} is isomorphic to the dicyclic ternary algebra $(A=J\oplus T,\bar{\ },*,\{...\})$ with $\bar a=a$, $\bar x=-x$, for any $a\in J$ and $x\in T$, and with
\begin{equation}\label{eq:JTdicyclic}
\begin{split}
a*b&=0,\\
a*x&=-x*a=-a\bullet x,\\
x*y&=-2\langle x\vert y\rangle,\\
\{a,x,w\}&=0=\{x,a,w\},\\
\{a,b,c\}&=-2\bigl((a\cdot b)\cdot c+a\cdot(b\cdot c)-(a\cdot c)\cdot b\bigr),\\
\{a,b,x\}&=b\bullet(a\bullet x),\\
\{x,y,a\}&=-2\langle a\bullet x\vert y\rangle,\\
\{x,y,z\}&=2\ptriple{x,y,z},
\end{split}
\end{equation}
for any $a,b,c\in J$, $x,y,z\in T$ and $w\in J\oplus T$. \qed
\end{proposition}

\medskip

Let $(J,T)$ be a $J$-ternary algebra, and let $A=J\oplus T$ be the associated dicyclic ternary algebra as in Proposition \ref{pr:JTdicyclic}. Due to \eqref{eq:JTdicyclic}, the unity $1$ of the Jordan algebra $J$ verifies the following conditions:
\begin{equation}\label{eq:11ax}
\begin{split}
&\{1,1,a\}=-2a=\{a,1,1\},\quad \{1,1,x\}=x,\quad\{x,1,1\}=0,\\
&\{a,1,b\}=-2a\cdot b,\quad 1*a=0,\quad 1*x=-x,
\end{split}
\end{equation}
for any $a\in J$ and $x\in T$.

Our next purpose is to show that the existence of an element with these properties characterizes the dicyclic ternary algebras coming from a $J$-ternary algebra. Actually a stronger result (Corollary \ref{co:dicyclicJT}) will be proved.

\medskip

Given a dicyclic ternary algebra $(A,\bar{\ },*,\{...\})$, the automorphism $x\mapsto \bar x$ induces a $\bZ_2$-grading: $A=A_0\oplus A_1$, where $A_0=\{x\in A:\bar x=x\}$ and $A_1=\{x\in A:\bar x=-x\}$. Usually the elements of $A_0$ will be denoted by $a,b,...$, while the elements of $A_1$ by $x,y,...$ in what follows.

\smallskip

\begin{lemma}\label{le:A*A0}
Let $(A,\bar{\ },*,\{...\})$ be a dicyclic ternary algebra which contains an element $e\in A_0$ such that $\{e,e,e\}=-2e$ and $\tau(e,e)\,(=\{e,e,.\})$ acts as a scalar on both $A_0$ and $A_1$. Then $\{e,e,a\}=-2a$ for any $a\in A_0$, $A_0*A_0=0$ and either $A*A=0$ or $\{e,e,x\}=x$ for any $x\in A_1$.
\end{lemma}
\begin{proof}
As $\{e,e,e\}=-2e$, our hypotheses imply $\{e,e,a\}=-2a$ for any $a\in A_0$. Let $\alpha\in\bF$ such that $\{e,e,x\}=\alpha x$ for any $x\in A_1$. Now equation \textbf{(D2)} implies $\{e,e,s*t\}+\{e,e,s\}*t+s*\{e,e,t\}=0$ for any $s,t\in A$. For $s=a\in A_0$ and $t=b\in A_0$ we get $-6a*b=0$, so that $A_0*A_0=0$. For $s=a\in A_0$ and $t=x\in A_1$ we get $(2\alpha-2)a*x=0$, so either $\alpha=1$ or $A_0*A_1=0$. Finally, for $s=x\in A_1$ and $t=y\in A_1$, we obtain $(2-2\alpha)x*y=0$, so either $\alpha=1$ or $A_1*A_1=0$. Therefore either $\alpha=1$ or $A*A=0$ as required.
\end{proof}

\begin{lemma}\label{le:A0A1A}
Let $(A,\bar{\ },*,\{...\})$ be a dicyclic ternary algebra which contains an element $e\in A_0$ such that $\{e,e,a\}=-2a$ for any $a\in A_0$, and $\{e,e,x\}=x$ and $\{x,e,e\}=0$ for any $x\in A_1$. Then $\{A_0,A_1,A\}=0=\{A_1,A_0,A\}$.
\end{lemma}
\begin{proof}
Note first that for $x\in A_1$,
\[
\begin{split}
\{e,x,e\}&=\{e,e,\{e,x,e\}\}\\
    &=\{\{e,e,e\},x,e\}-\{e,\{e,e,x\},e\}+\{e,x,\{e,e,e\}\}\quad\text{using \textbf{(D5)}}\\
    &=-2\{e,x,e\}-\{e,x,e\}-2\{e,x,e\}=-5\{e,x,e\},
\end{split}
\]
so we get $\{e,A_1,e\}=0$.

For $a\in A_0$ and $x\in A_1$, \textbf{(D3)} gives $\{e,a*x,e\}+\{a,x*e,e\}+\{x,e*a,e\}=0$, but $e*a=0$ (Lemma \ref{le:A*A0}), and $\{e,a*x,e\}\in\{e,A_1,e\}=0$. We conclude that $\{A_0,A_1*e,e\}=0$ holds. But \textbf{(D1)} gives $x=\{e,e,x\}-\{x,e,e\}=-(e*x)*e=(x*e)*e$. Hence $A_1*e=A_1$ and $\{A_0,A_1,e\}=0$.

Now \textbf{(D4)} gives $\{a*x,e,e\}-\{x*e,a,e\}+\{e*a,x,e\}=0$, but $\{a*x,e,e\}\in\{A_1,e,e\}=0$ and $e*A_0\in A_0*A_0=0$, so we conclude $\{A_1,A_0,e\}=0$.

Therefore, $\{A_0,A_1,e\}=\{A_1,A_0,e\}=0$. Finally, \textbf{(D5)} gives
\[
[\tau(e,e),\tau(a,x)]=\tau\bigl(\{e,e,a\},x\bigr)-\tau\bigl(a,\{e,e,x\}\bigr)=-3\tau(a,x),
\]
but also
\[
[\tau(e,e),\tau(a,x)]=-[\tau(a,x),\tau(e,e)] =-\tau\bigl(\{a,x,e\},e\bigr)-\tau\bigl(e,\{x,a,e\}\bigr)=0.
\]
Thus $\tau(a,x)=0$, and in the same vein we get $\tau(x,a)=0$, as required.
\end{proof}

\smallskip

Under the conditions of the previous Lemma, the Lie algebra $\frg(A)=\tau(A,A)\oplus \iota_1(A)\oplus\iota_2(A)$ in Theorem \ref{th:conversedicyclic} satisfies that its subalgebra $\tau(A,A)=\tau(A_0,A_0)+\tau(A_1,A_1)$ contains the elements
\[
H=\tau(e,e),\quad E=-\iota_2(e),\quad F=\iota_1(e),
\]
which satisfy $[E,F]=H$, $[H,E]=2E$, and $[H,F]=-2F$, and hence span a subalgebra isomorphic to $\frsl_2(\bF)$. Moreover, the following conditions hold:
\[
\begin{aligned}
\null [E,\iota_1(A_1)]&=-\tau(A_1,e)=0,&\qquad [F,\iota_2(A_1)]&=\tau(e,A_1)=0,\\
 [F,\iota_1(x)]&=\iota_2(e*x),&\qquad [E,\iota_2(x)]&=\iota_1(e*x),\\
 [H,\iota_1(x)]&=\iota_1(x),&\qquad[H,\iota_2(x)]&=-\iota_2(x),
\end{aligned}
\]
for any $x\in A_1$.

Hence for any $0\ne x\in A_1$, the vector space $\bF\iota_1(x)+\bF\iota_2(e*x)$ is a two-dimensional irreducible module for our copy of $\frsl_2(\bF)$, since $e*(e*x)=x$.

Also, we have $[E,\iota_2(A_0)]=0=[F,\iota_1(A_0)]$ as $e*A_0=0$, and therefore, for any $0\ne a\in A_0$, the vector space $\bF\iota_1(a)+\bF\tau(a,e)+\bF\iota_2(\{e,a,e\})$ is isomorphic to the adjoint module for $\frsl_2(\bF)$ under the map $F\mapsto \iota_1(a)$, $H\mapsto \tau(a,e)$ and $E\mapsto \frac{1}{2}\iota_2(\{e,a,e\})$, because $[E,\iota_1(a)]=-[\iota_2(e),\iota_1(a)]=\tau(a,e)$, and $[E,\tau(a,e)]=[\tau(a,e),\iota_2(e)]=-\iota_2(\{e,a,e\}$ (while $[E,H]=-2E$). Note that for any $a\in A_0$, \textbf{(D5)} gives $4a=\{a,e,\{e,e,e\}\}=\{\{a,e,e\},e,e\}-\{e,\{e,a,e,\},e\} +\{e,e,\{a,e,e\}\}=8a-\{e,\{e,a,e\},e\}$, so that $\{e,\{e,a,e\},e\}=4a$, and in particular $\{e,a,e\}\ne 0$ if $a\ne 0$.

We summarize these arguments in the following result:

\begin{lemma}\label{le:adjointandtwo}
Let $(A,\bar{\ },*,\{...\})$ be a dicyclic ternary algebra which contains an element $e\in A_0$ such that $\{e,e,a\}=-2a$ for any $a\in A_0$, and $\{e,e,x\}=x$ and $\{x,e,e\}=0$ for any $x\in A_1$. Then
\begin{itemize}
\item For any $0\ne a\in A_0$, the vector space $\bF\iota_1(a)+\bF\tau(a,e)+\bF\iota_2(\{e,a,e\})$ is isomorphic to the adjoint module for $\frsl_2(\bF)$ under the map $F\mapsto \iota_1(a)$, $H\mapsto \tau(a,e)$ and $E\mapsto \frac{1}{2}\iota_2(\{e,a,e\})$
\item For any $0\ne x\in A_1$, the vector space $\bF\iota_1(x)+\bF\iota_2(e*x)$ is a two-dimensional irreducible module for $\frsl_2(\bF)$. \qed
\end{itemize}
\end{lemma}

\smallskip

An extra previous result is still needed:

\smallskip

\begin{lemma}\label{le:tauAACentr}
Let $(A,\bar{\ },*,\{...\})$ be a dicyclic ternary algebra which contains an element $e\in A_0$ such that $\{e,e,a\}=-2a$ for any $a\in A_0$, and $\{e,e,x\}=x$ and $\{x,e,e\}=0$ for any $x\in A_1$. Then $\tau(A,A)=\tau(A_0,e)\oplus \Cent_{\frg(A)}\bigl(\frsl_2(\bF)\bigr)$ (where $\Cent_{\frg(A)}\bigl(\frsl_2(\bF)\bigr)$ denotes the centralizer in $\frg(A)$ of our subalgebra $\frsl_2(\bF)=\bF H\oplus\bF E\oplus\bF F$).
\end{lemma}
\begin{proof}
The element $H=\tau(e,e)$ acts diagonally with eigenvalues $-2$ on $\iota_1(A_0)$, $2$ on $\iota_2(A_0)$, $1$ on $\iota_1(A_1)$, $-1$ on $\iota_2(A_1)$ and $0$ on $\tau(A,A)=\tau(A_0,A_0)+\tau(A_1,A_1)$ (see Theorem \ref{th:conversedicyclic} for the bracket on the Lie algebra $\frg(A)$). In particular we have
\[
\Cent_{\frg(A)}\bigl(\frsl_2(\bF)\bigr)\subseteq\{X\in\frg(A):[H,X]=0\}=\tau(A,A).
\]
Also note that for $0\ne a\in A_0$ and $0\ne x\in A_1$, $[F,\tau(a,e)]=-[\tau(a,e),\iota_1(e)]=-\iota_1(\{a,e,e\})=2\iota_1(a)\ne 0$, so we get
\begin{equation}\label{eq:tauA0ecapCentr}
\tau(A_0,e)\cap \Cent_{\frg(A)}\bigl(\frsl_2(\bF)\bigr)=0.
\end{equation}

Now, for $x,y\in A_1$, \textbf{(D4)} gives
\[
-\tau(x*e,y)-\tau(e*y,x)+\tau(y*x,e)=0,
\]
which we rewrite as
\[
\tau(x*y,e)=\tau(e*x,y)-\tau(e*y,x).
\]
Hence we get
\begin{equation}\label{eq:tauexy}
\tau(e*x,y)=\frac{1}{2}\tau(x*y,e)+\frac{1}{2}\bigl(\tau(e*x,y)+\tau(e*y,x)\bigr),
\end{equation}
and
\begin{equation}\label{eq:tauexyeyx}
\begin{split}
[\tau(e*x,y)+\tau(e*y,x),\iota_1(e)]&=\iota_1\bigl(\{e*x,y,e\}+\{e*y,x,e\}\bigr)\\
[\tau(e*x,y)+\tau(e*y,x),\iota_2(e)]&=\iota_2\bigl(\{y,e*x,e\}+\{x,e*y,e\}\bigr).
\end{split}
\end{equation}
But $\tau(A_0,A_1)=0$ (Lemma \ref{le:A0A1A}) and \textbf{(D1)} gives:
\[
\{e*x,y,e\}=((e*x)*e)*y=-\{x,e,e\}*y=-x*y,
\]
so
\[
\{e*x,y,e\}+\{e*y,x,e\}=-(x*y+y*x)=0.
\]
In a similar way we get $\{x,e*y,e\}=(x*e)*(e*y)$, so
\[
\{x,e*y,e\}+\{y,e*x,e\}=-\bigl((e*x)*(e*y)+(e*y)*(e*x)\bigr)=0.
\]
It follows now from \eqref{eq:tauexyeyx} that
\[
[\tau(e*x,y)+\tau(e*y,x),E]=0=[\tau(e*x,y)+\tau(e*y,x),F],
\]
so that the element $\tau(e*x,y)+\tau(e*y,x)$ is in $\Centr_{\frg(A)}\bigl(\frsl_2(\bF)\bigr)$, and hence \eqref{eq:tauexy} shows that
\begin{equation}\label{eq:tauA1A1}
\tau(A_1,A_1)=\tau(e*A_1,A_1)\subseteq \tau(A_0,e)+\Centr_{\frg(A)}\bigl(\frsl_2(\bF)\bigr).
\end{equation}
Finally, for any $a,b\in A_0$ we have:
\[
\begin{split}
[\tau(a,e),\tau(b,e)]&=\tau(\{a,e,b\},e)-\tau(b,\{e,a,e\}),\\
[\tau(a,e),\tau(b,e)]&=-[\tau(b,e),\tau(a,e)]=-\tau(\{b,e,a\},e)+\tau(a,\{e,b,e\}).
\end{split}
\]
Hence, as $\{a,e,b\}=\{b,e,a\}$ by \textbf{(D1)} (as $A_0*A_0=0$ by Lemma \ref{le:A*A0}), we get
\[
2\tau(\{a,e,b\},e)=\tau(a,\{e,b,e\})+\tau(b,\{e,a,e\}),
\]
and
\begin{equation}\label{eq:tauaebe}
\tau(a,\{e,b,e\})=\tau(\{a,e,b\},e) +\frac{1}{2}\bigl(\tau(a,\{e,b,e\}-\tau(b,\{e,a,e\}\bigr).
\end{equation}
Also,
\[
\begin{split}
-2\{a,e,b\}&=\{a,e,\{b,e,e\}\}\\
    &=\{\{a,e,b\},e,e\}-\{b,\{e,a,e\},e\}+\{b,e,\{a,e,e\}\}\\
    &=-2\{a,e,b\}-\{b,\{e,a,e\},e\}-2\{b,e,a\}.
\end{split}
\]
Hence $\{b,\{e,a,e\},e\}=2\{b,e,a\}$, which is symmetric on $a,b$ by \textbf{(D1)}. Therefore,
\[
\begin{split}
[\tau(a,\{e,b,e\})-\tau(b,\{e,a,e\}),F]
    &=[\tau(a,\{e,b,e\})-\tau(b,\{e,a,e\}),\iota_1(e)]\\
    &=\iota_1\bigl(\{a,\{e,b,e\},e\}-\{b,\{e,a,e\},e\}\bigr)=0.
\end{split}
\]
And because of \textbf{(D1)} and \textbf{(D5)} we have:
\[
\begin{split}
\{\{e,a,e\},b,e\}&=\{e,b,\{e,a,e\}\}\\
&=\{\{e,b,e\},a,e\}-\{e,\{b,e,a\},e\}+\{e,a,\{e,b,e\}\}\\
&=2\{\{e,b,e\},a,e\}-\{e,\{b,e,a\},e\}.
\end{split}
\]
Thus $2\{\{e,b,e\},a,e\}-\{\{e,a,e\},b,e\}=\{e,\{b,e,a\},e\}$ is symmetric on $a,b$. Adding the symmetric element $\{\{e,b,e\},a,e\}+\{\{e,a,e\},b,e\}$ we obtain that $\{\{e,b,e\},a,e\}$ is symmetric on $a,b$. Then we get:
\[
\begin{split}
[\tau(a,\{e,b,e\})-\tau(b,\{e,a,e\}),E]
    &=-[\tau(a,\{e,b,e\})-\tau(b,\{e,a,e\}),\iota_2(e)]\\
    &=\iota_2\bigl(\{\{e,b,e\},a,e\}-\{\{e,a,e\},b,e\}\bigr)=0.
\end{split}
\]
Therefore $\tau(a,\{e,b,e\})-\tau(b,\{e,a,e\})$ belongs to $\Centr_{\frg(A)}\bigl(\frsl_2(\bF)\bigr)$, and equation \eqref{eq:tauaebe} gives:
\begin{equation}\label{eq:tauA0A0}
\tau(A_0,A_0)=\tau(A_0,\{e,A_0,e\})\subseteq \tau(A_0,e)+\Centr_{\frg(A)}\bigl(\frsl_2(\bF)\bigr).
\end{equation}

Equations \eqref{eq:tauA0ecapCentr}, \eqref{eq:tauA1A1} and \eqref{eq:tauA0A0} give the desired result.
\end{proof}

\medskip

\begin{theorem}\label{th:gABC1C1}
Let $(A,\bar{\ },*,\{...\})$ be a dicyclic ternary algebra which contains an element $e\in A_0$ such that $\{e,e,a\}=-2a$ for any $a\in A_0$, and $\{e,e,x\}=x$ and $\{x,e,e\}=0$ for any $x\in A_1$. Then the attached Lie algebra $\frg(A)$ in Theorem \ref{th:conversedicyclic} is a $BC_1$-graded Lie algebra of type $C_1$.
\end{theorem}
\begin{proof}
The two previous lemmas show that we have a decomposition:
\[
\frg(A)=\Bigl(\tau(A_0,e)\oplus\iota_1(A_0)\oplus\iota_2(A_0)\Bigr)\oplus
\Bigl(\iota_1(A_1)\oplus\iota_2(A_1)\Bigr)\oplus \Centr_{\frg(A)}\bigl(\frsl_2(\bF)\bigr),
\]
where
\begin{itemize}
\item $\tau(A_0,e)\oplus\iota_1(A_0)\oplus\iota_2(A_0)$ is a direct sum of copies of the adjoint module for $\frsl_2(\bF)=\bF H+\bF E+\bF F$,
\item $\iota_1(A_1)\oplus\iota_2(A_1)$ is a direct sum of copies of the two-dimensional irreducible module for $\frsl_2(\bF)$,
\item $\Centr_{\frg(A)}\bigl(\frsl_2(\bF)\bigr)$ is a trivial module for $\frsl_2(\bF)$,
\end{itemize}
and this proves the Theorem.
\end{proof}

\medskip

Therefore, under the hypotheses of this Theorem, the Lie algebra $\frg(A)$ can be identified to
\[
\bigl(\frsl(V)\otimes A_0\bigr)\oplus\bigl(V\otimes A_1\bigr)\oplus \frd,
\]
where we fix a symplectic basis $\{u,v\}$ of $V$, take $H,E,F$ the endomorphisms of $V$ with coordinate matrices given in \eqref{eq:matricesHEF}, and identify elements as follows:
\[
\begin{aligned}
H\otimes a&\leftrightarrow \tau(a,e)\\
E\otimes a&\leftrightarrow\frac{1}{2}\iota_2(\{e,a,e\})\\
F\otimes a&\leftrightarrow \iota_1(a)
\end{aligned}
\qquad
\begin{aligned}
u\otimes x&\leftrightarrow \iota_1(x)\\
v\otimes x&\leftrightarrow \iota_2(e*x)
\end{aligned}
\]
for any $a\in A_0$ and $x\in A_1$. Also we take the subalgebra $\frd=\Centr_{\frg(A)}\bigl(\frsl_2(\bF)\bigr)= \Centr_{\tau(A,A)}\bigl(\frsl_2(\bF)\bigr)$.

The Lie bracket in $\frg(A)$ then induces some operations as in \eqref{eq:g[]} and \eqref{eq:maps} on $A_0$ and $A_1$, which we determine now. We take $a,b\in A_0$ and $x,y,z\in A_1$.

\begin{itemize}
\item $[H\otimes a,F\otimes b]=-2F\otimes a\cdot b$, but $[\tau(a,e),\iota_1(b)]=\iota_1(\{a,e,b\})$. Hence,
    \[
    a\cdot b=-\frac{1}{2}\{a,e,b\}.
    \]

\item $[H\otimes a,u\otimes x]=u\otimes a\bullet x$, and $[\tau(a,e),\iota_1(x)]=\iota_1(\{a,e,x\})$, so
    \[
    a\bullet x=\{a,e,x\}.
    \]

\item $[u\otimes x,v\otimes y]=-H\otimes \langle x\vert y\rangle + d_{x,y}$, and $[\iota_1(x),\iota_2(e*y)]=\tau(x,e*y)$. Since $(x*e)*e=-\{x,e,e\}+\{e,e,x\}=x$ (see \textbf{(D1)}), equation \eqref{eq:tauexy} gives
    \[
    \begin{split}
    \tau(x,e*y)&=\tau(e*(e*x),e*y)\\
       &=\frac{1}{2}\tau((e*x)*(e*y),e)+\frac{1}{2}\bigl(\tau(x,e*y)+\tau(y,e*x)\bigr),
    \end{split}
    \]
    so we get:
    \[
    \left\{\begin{aligned}
    \langle x\vert y\rangle &=-\frac{1}{2}(e*x)*(e*y),\\
    d_{x,y}&=\frac{1}{2}\bigl(\tau(x,e*y)+\tau(y,e*x)\bigr),
    \end{aligned}\right.
    \]
    and, because of \eqref{eq:(xyz)ds},
    \[
    \begin{split}
    \ptriple{x,y,z}&=\frac{1}{4}\bigl(-\{x,e*y,z\}-\{y,e*x,z\}-\{(e*x)*(e*y),e,z\}\bigr)\\
    &=-\frac{1}{4}\bigl(\{x,e*y,z\}+\{y,e*x,z\}\\
    &\qquad\qquad\qquad +\{(e*y)*e,e*x,z\}+\{e*(e*x),e*y,z\}\bigr)\\
    &=-\frac{1}{4}\bigl(2\{x,e*y,z\}\bigr)=-\frac{1}{2}\{x,e*y,z\},
    \end{split}
    \]
    where \textbf{(D1)} and \textbf{(D4)} have been used.
\end{itemize}

This proves the following result:

\begin{corollary}\label{co:dicyclicJT}
Let $(A,\bar{\ },*,\{...\})$ be a dicyclic ternary algebra which contains an element $e\in A_0$ such that $\{e,e,a\}=-2a$ for any $a\in A_0$, and $\{e,e,x\}=x$ and $\{x,e,e\}=0$ for any $x\in A_1$. Then the pair $(A_0,A_1)$ is a $J$-ternary algebra with operations:
\[
\begin{split}
a\cdot b&=-\frac{1}{2}\{a,e,b\},\\
a\bullet x&=\{a,e,x\},\\
\langle x\vert y\rangle&=-\frac{1}{2}(e*x)*(e*y),\\
\ptriple{x,y,z}&=-\frac{1}{2}\{x,e*y,z\},
\end{split}
\]
for any $a,b\in A_0$ and $x,y,z\in A_1$. \qed
\end{corollary}

\medskip

\begin{remark}\label{re:dicyclicJT}
Had we started with the dicyclic ternary algebra $A=J\oplus T$ attached to a $J$-ternary algebra, with $e=1\in J$ as in Proposition \ref{pr:JTdicyclic}, then the element $e$ would satisfy, besides the conditions $\{e,e,a\}=-2a$ for any $a\in A_0=J$ and $\{e,e,x\}=x$, $\{x,e,e\}=0$ for any $x\in A_1=T$, the extra condition $e*x=-1\bullet x=-x$ for any $x\in T$ in \eqref{eq:11ax}, and hence we would have obtained
\[
\begin{split}
a\cdot b&=-\frac{1}{2}\{a,e,b\},\\
a\bullet x&=\{a,e,x\}=-(a*x)*e=-a*x\quad\text{(using \textbf{(D1)}),}\\
\langle x\vert y\rangle&=-\frac{1}{2}x*y,\\
\ptriple{x,y,z}&=\frac{1}{2}\{x,y,z\},
\end{split}
\]
for $a,b\in J$ and $x,y,z\in T$, which recover the original operations in $J$ and $T$ (see \eqref{eq:JTdicyclic}).
\end{remark}

\medskip

\section{$(\epsilon,\delta)$ Freudenthal-Kantor triple systems and dicyclic algebras}\label{se:edFKTSdicyclic}

The results in Section \ref{se:dicyclicJternary} show that for a $J$-ternary algebra $(J,T)$, the Lie algebra $\frg(J,T)$ is a Lie algebra with dicyclic symmetry, because the dicyclic group is contained in the symplectic group $Sp(V)$ which acts by automorphisms on $\frg(J,T)$. In particular (see Section \ref{se:JternarySpecialFKTS}) given a special $(1,1)$ Freudenthal-Kantor triple system $(U,xyz)$, the pair $(J,U)$ is endowed with a structure of  $J$-ternary algebra (Theorem \ref{th:JT11}), where $J$ is the special Jordan algebra $\bF 1+K(U,U)$, and hence we get that the corresponding Lie algebra $\frg(J,U)$ is a Lie algebra with dicyclic symmetry (that is, there is an action of the dicyclic group $\Dic_3$ by automorphisms on $\frg(J,U)$).

\smallskip

In this section, given an arbitrary $(\epsilon,\delta)$ Freudenthal-Kantor triple system $(U,xyz)$, a $5$-graded Lie algebra, for $\delta=1$, or Lie superalgebra, for $\delta=-1$, with dicyclic symmetry will be defined. For a special $(1,1)$ Freudenthal-Kantor triple system this algebra is quite close to our previous $\frg(J,U)$.

The construction of $\frg(U)$ is based on the pioneering work of Yamaguti and Ono \cite{YO84}.

\smallskip

Thus, let $(U,xyz)$ be a $(\epsilon,\delta)$ Freudenthal-Kantor triple system, then the space of $2\times 1$ matrices over $U$:
\[
\cT=\left\{\begin{pmatrix}x\\ y\end{pmatrix}: x,y\in U\right\}
\]
becomes a Lie triple system for $\delta=1$ and an anti-Lie triple system for $\delta=-1$ (see \cite[Section 3]{YO84}) by means of the triple product:
\begin{equation}\label{eq:Lietriple}
\begin{split}
&\left[\begin{pmatrix}a_1\\ b_1\end{pmatrix}\begin{pmatrix} a_2\\ b_2\end{pmatrix}\begin{pmatrix} a_3\\ b_3\end{pmatrix}\right]\\
&\qquad =
\begin{pmatrix} L(a_1,b_2)-\delta L(a_2,b_1)&\delta K(a_1,a_2)\\
-\epsilon K(b_1,b_2)&\epsilon L(b_2,a_1)-\epsilon\delta L(b_1,a_2)\end{pmatrix}\begin{pmatrix} a_3\\ b_3\end{pmatrix}
\end{split}
\end{equation}
and, therefore, the vector space
\begin{equation}\label{eq:L}
\cL=\espan{\begin{pmatrix} L(a,b)&K(c,d)\\ K(e,f)&\epsilon L(b,a)\end{pmatrix} : a,b,c,d,e,f\in U}
\end{equation}
is a Lie subalgebra of $\Mat_2\bigl(\End_\bF(U)\bigr)^-$. (Given an associative algebra $A$, $A^-$ denotes the Lie algebra defined on $A$ with product given by the usual Lie bracket $[x,y]=xy-yx$.)

Hence we get either a $\bZ_2$-graded Lie algebra (for $\delta=1$) or a superalgebra (for $\delta=-1$)
\begin{equation}\label{eq:gU}
\frg(U)=\cL\oplus\cT
\end{equation}
where $\cL$ is the even part and $\cT$ the odd part. The bracket in $\frg(U)$ is given by:
\begin{itemize}
\item the given bracket in $\cL$ as a subalgebra of $\Mat_2\bigl(\End_\bF(U)\bigr)^-$,
\item $[M,X]=M(X)$ for any $M\in \cL$ and $X\in \cT$ (note that $\Mat_2\bigl(\End_\bF(U)\bigr)\simeq \End_\bF(\cT)$),
\item for any $a_1,a_2,b_1,b_2\in U$:
\[
\left[\begin{pmatrix}a_1\\ b_1\end{pmatrix},\begin{pmatrix}a_2\\ b_2\end{pmatrix}\right]=\begin{pmatrix} L(a_1,b_2)-\delta L(a_2,b_1)&\delta K(a_1,a_2)\\
-\epsilon K(b_1,b_2)&\epsilon L(b_2,a_1)-\epsilon\delta L(b_1,a_2)\end{pmatrix}.
\]
\end{itemize}

\smallskip

This (super)algebra $\frg(U)$ is consistently $\bZ$-graded as follows:
\[
\begin{split}
\frg(U)_{(0)}&=\espan{\begin{pmatrix} L(a,b)&0\\ 0&\epsilon L(b,a)\end{pmatrix} : a,b\in U},\\
\frg(U)_{(1)}&=\begin{pmatrix}U\\ 0\end{pmatrix},\\
\frg(U)_{(-1)}&=\begin{pmatrix} 0\\ U\end{pmatrix},\\
\frg(U)_{(2)}&=\espan{\begin{pmatrix} 0&K(a,b)\\ 0&0\end{pmatrix} : a,b\in U},\\
\frg(U)_{(-2)}&=\espan{\begin{pmatrix} 0&0\\ K(a,b)&0\end{pmatrix} : a,b\in U},
\end{split}
\]
so that $\frg(U)$ is $5$-graded and
\[
\cL=\frg(U)_{(-2)}\oplus\frg(U)_{(0)}\oplus\frg(U)_{(2)},\qquad \cT=\frg(U)_{(-1)}\oplus\frg(U)_{(1)}.
\]

\smallskip

On the other hand, if $U$ is a special $(\epsilon,\epsilon)$ Freudenthal-Kantor triple system, Corollaries \ref{co:11Lie} and \ref{co:-1-1Lie} give another construction of a Lie (super)algebra $\frg(J,U)$ attached to $U$ (Lie algebra for $\epsilon=1$ and superalgebra for $\epsilon=-1$), where $J=\bF 1+K(U,U)$:
\[
\frg(J,U)=\bigl(\frsl(V)\otimes J\bigr)\oplus\bigl(V\otimes U\bigr)\oplus\frd,
\]
with $\frd=S(U,U)$ (the linear span of the operators $S(x,y)$ defined in \eqref{eq:STs}).

\begin{proposition}
Let $(U,xyz)$ be a special $(\epsilon,\epsilon)$ Freudenthal-Kantor triple system, and let $\{u,v\}$ be a fixed symplectic basis of the two-dimensional vector space $V$ (so that $(u\vert v)=1$ for a fixed nonzero skew symmetric bilinear map $(.\vert.)$). Then the linear map $\frg(U)\rightarrow \frg(J,U)$ given by:
\[
\begin{split}
\begin{pmatrix}a\\ b\end{pmatrix}\in\cT&\mapsto u\otimes a+v\otimes b\in V\otimes U,\\
\begin{pmatrix} L(a,b)&0\\ 0&\epsilon L(b,a)\end{pmatrix}&\mapsto -\epsilon\gamma_{u,v}\otimes K(a,b)\, -\, S(a,b)\ \in\bigl(\frsl(V)\otimes J\bigr)\oplus\frd,\\
\begin{pmatrix} 0&K(a,b)\\ 0&0\end{pmatrix}&\mapsto -\gamma_{u,u}\otimes K(a,b)\,\in\frsl(V)\otimes J,\\
\begin{pmatrix} 0&0\\ K(a,b)&0\end{pmatrix}&\mapsto \gamma_{v,v}\otimes K(a,b)\,\in\frsl(V)\otimes J,
\end{split}
\]
for any $a,b\in U$, is a one-to-one Lie (super)algebra homomorphism.
\end{proposition}
\begin{proof} This is done by straightforward computations.
\end{proof}

Under this homomorphism $\frg(U)_{(1)}$ (respectively $\frg(U)_{(-1)}$)  is identified to $u\otimes U$ (respectively $v\otimes U$).

\smallskip

In particular, for a special $(\epsilon,\epsilon)$ Freudenthal-Kantor triple system, the Lie (super)algebra $\frg(U)$ inherits the dicyclic symmetry from $\frg(J,U)$ considered in Section \ref{se:JternarySpecialFKTS} (assuming the primitive cubic roots $\omega$ and $\omega^2$ of $1$ are contained in the ground field $\bF$).

Our next purpose is to check that for any arbitrary $(\epsilon,\delta)$ Freudenthal-Kantor triple system $(U,xyz)$ the Lie (super)algebra $\frg(U)$ presents dicyclic symmetry (assuming $\omega\in\bF$), which restricts to the previous one for the special $(\epsilon,\epsilon)$ Freudenthal-Kantor triple systems.

First, the fact that $\frg(U)$ is $\bZ$-graded gives a group homomorphism:
\[
\begin{split}
\bF^\times&\longrightarrow \Aut\bigl(\frg(U)\bigr)\\
\mu&\mapsto \phi_\mu:X\mapsto \mu^iX,\ \text{for $X\in\frg(U)_{(i)}$, $-2\leq i\leq 2$.}
\end{split}
\]
Also, the map
\[
\begin{split}
\theta:\cT&\longrightarrow \cT\\
\begin{pmatrix}a\\ b\end{pmatrix}&\mapsto \,\begin{pmatrix}-\epsilon b\\ \delta a\end{pmatrix}\ \text{for $a,b\in U$,}
\end{split}
\]
is an automorphism of the (anti)Lie triple system $\cT$, and hence extends naturally to an automorphism, also denoted by $\theta$, of $\frg(U)=\cL\oplus\cT$, with $\theta(M)=\theta M\theta^{-1}$ for any $M\in \cL\subseteq\End_\bF(\cT)$.

Note that $\theta^2=1$ for $\epsilon=-\delta$, while $\theta^4=1\ne \theta^2$ for $\epsilon=\delta$. For $\epsilon=\delta$, $\theta^2$ is $-1$ on $\cT$ and $1$ on $\cL$, so it coincides with the $\bZ_2$-grading automorphism of $\frg(U)$. A few computations give precise formulas for the action of $\theta$:
\[
\begin{split}
\theta\begin{pmatrix} L(a,b)&0\\ 0&\epsilon L(b,a)\end{pmatrix}
  &=\left[\theta\begin{pmatrix}a\\ 0\end{pmatrix},\theta\begin{pmatrix} 0\\ b\end{pmatrix}\right]\\
  &=-\epsilon\delta\left[\begin{pmatrix}0\\ a\end{pmatrix},\begin{pmatrix}b\\ 0\end{pmatrix}\right]=\begin{pmatrix} \epsilon L(b,a)&0\\ 0& L(a,b)\end{pmatrix},\\[6pt]
\theta\begin{pmatrix} 0&K(a,b)\\ 0&0\end{pmatrix}
  &=\delta\left[\theta\begin{pmatrix}a\\ 0\end{pmatrix},\theta\begin{pmatrix} b\\ 0\end{pmatrix}\right]\\
  &=\delta\left[\begin{pmatrix}0\\ a\end{pmatrix},\begin{pmatrix}0\\ b\end{pmatrix}\right]=-\epsilon\delta\begin{pmatrix} 0&0\\ K(a,b)&0\end{pmatrix},\\[6pt]
\theta\begin{pmatrix} 0&0\\ K(a,b)&0\end{pmatrix}
  &=-\epsilon\left[\theta\begin{pmatrix}0\\ a\end{pmatrix},\theta\begin{pmatrix} 0\\ b\end{pmatrix}\right]\\
  &=-\epsilon\left[\begin{pmatrix}a\\ 0\end{pmatrix},\begin{pmatrix}b\\ 0\end{pmatrix}\right]
  =-\epsilon\delta\begin{pmatrix} 0&K(a,b)\\ 0&0\end{pmatrix}.
\end{split}
\]

Assuming $\omega\in\bF$, the automorphisms $\phi=\phi_\omega$ and $\theta$ satisfy \eqref{eq:theta4phi3}. The action of $\phi$ induces a $\bZ_3$-grading:
$\frg(U)=\frg(U)_1\oplus\frg(U)_\omega\oplus\frg(U)_{\omega^2}$, with $\frg(U)_1=\frg(U)_{(0)}$, $\frg(U)_\omega=\frg(U)_{(-2)}\oplus\frg(U)_{(1)}\simeq K(U,U)\oplus U$, and $\frg(U)_{\omega^2}=\frg(U)_{(2)}\oplus\frg(U)_{(-1)}\simeq K(U,U)\oplus U$

\begin{remark}\label{re:gUS3}
Actually, for $\delta=-\epsilon$ we get $\theta^2=1$, and hence $\phi$ and $\theta$ generate a subgroup of automorphisms of $\frg(U)$ isomorphic to the symmetric group $S_3$. In particular, for $\delta=-\epsilon=1$, $\frg(U)$ is a Lie algebra with $S_3$-symmetry, and this defines a structure of \emph{generalized Malcev algebra} (see \cite{EOS3S4}) on $\frg_\omega\simeq K(U,U)\oplus U$. This extends \cite[Proposition 4.1]{EOS3S4}.
\end{remark}

\smallskip

For $\epsilon=\delta=1$, $\frg(U)$ becomes a Lie algebra with dicyclic symmetry, and this shows that $A=\frg(U)_\omega\simeq K(U,U)\oplus U$ becomes a dicyclic ternary algebra (see Section \ref{se:dicyclic}), with $A_0=\frg(U)_{(-2)}\simeq K(U,U)$ (fixed by $\theta^2$) and $A_1=\frg(U)_{(1)}\simeq U$. Identify $U$ with $A_1=\frg(U)_{(1)}$ and $K(U,U)$ with $A_0=\frg(U)_{(-2)}$ in the natural way:
\[
u\leftrightarrow \begin{pmatrix} u\\ 0\end{pmatrix},\qquad K(a,b)\leftrightarrow \begin{pmatrix}0&0\\ K(a,b)&0\end{pmatrix}.
\]

After these identifications, we can compute the binary and ternary products in \eqref{eq:barsquaretriple}.

First note that $A_0*A_0=0$ as $[\frg(U)_{(-2)},\frg(U)_{(-2)}]=0$. Also, for $a,b,x,y\in U$ we have:
\begin{gather*}
\theta^{-1}\left(\left[\begin{pmatrix}0&0\\ K(a,b)&0\end{pmatrix},\begin{pmatrix} 0\\ x\end{pmatrix}\right]\right)
 =\theta^{-1}\begin{pmatrix} 0\\ K(a,b)x\end{pmatrix}
=\begin{pmatrix} K(a,b)x\\ 0\end{pmatrix},\\[6pt]
\theta^{-1}\left(\left[\begin{pmatrix} x\\ 0\end{pmatrix},\begin{pmatrix} y\\ 0\end{pmatrix}\right]\right)
 =\theta^{-1}\begin{pmatrix} 0&K(x,y)\\ 0&0\end{pmatrix}
=-\begin{pmatrix} 0&0\\ K(x,y)&0\end{pmatrix}.
\end{gather*}

Therefore, the binary product in $A=K(U,U)\oplus U$ is given by:
\begin{equation}\label{eq:*productKUUU}
\begin{split}
M_1*M_2&=0,\\
M*x&=Mx,\\
x_1*x_2&=-K(x_1,x_2),
\end{split}
\end{equation}
for any $M,M_1,M_2\in K(U,U)$ and $x,x_1,x_2\in U$.

\smallskip

On the other hand we have:
\[
[\frg(U)_{(1)},\theta\bigl(\frg(U)_{(-2)}\bigr)]
=[\frg(U)_{(1)},\frg(U)_{(2)}]=0
=[\frg(U)_{(-2)},\theta\bigl(\frg(U)_{(1)}\bigr)],
\]
so we get the following instances of the triple product:
\[
\{A_0,A_1,A\}=0=\{A_1,A_0,A\}.
\]

\smallskip

Now, for any $a_1,a_2,b_1,b_2\in U$, we get
\[
\begin{split}
&\left[\begin{pmatrix} 0&0\\ K(a_1,b_1)&0\end{pmatrix},\theta\begin{pmatrix} 0&0\\ K(a_2,b_2)&0\end{pmatrix}\right]\\
 &\qquad\qquad\qquad =-\left[\begin{pmatrix} 0&0\\ K(a_1,b_1)&0\end{pmatrix},\begin{pmatrix} 0& K(a_2,b_2)\\ 0&0\end{pmatrix}\right]\\
 &\qquad\qquad\qquad =\begin{pmatrix} K(a_2,b_2)K(a_1,b_1) &0\\ 0&-K(a_1,b_1)K(a_2,b_2)\end{pmatrix}.
\end{split}
\]
The fact that this last matrix belongs to $\frg(U)_{(0)}$ is a consequence of the next result (see \cite[(2.9) and (2.10)]{YO84}).

\begin{lemma}\label{le:identitiesedFKTS}
Let $(U,xyz)$ be an $(\epsilon,\delta)$ Freudenthal-Kantor triple system, then for any $a,b,c,d\in U$ the following identities hold:
\begin{subequations}
\begin{align}
&\epsilon K(a,b)K(c,d)+L\bigl(K(a,b)c,d\bigr)-\delta L\bigl(K(a,b)d,c\bigr)=0,\label{eq:KKLK.}\\
&K(c,d)K(a,b)+\delta L\bigl(c,K(a,b)d\bigr)-L\bigl(d,K(a,b)c\bigr)=0.\label{eq:KKL.K}
\end{align}
\end{subequations}
\end{lemma}
\begin{proof}
Equation \eqref{eq:FK2} gives
\[
\begin{split}
K\bigl(K(u,v)x&,y\bigr)z-\delta K\bigl(K(u,v)z,y\bigr)x\\
 &=L(y,x)K(u,v)z-\delta L(y,z)K(u,v)x-\epsilon K(u,v)K(x,z)y,
\end{split}
\]
But \eqref{eq:LK} gives $K\bigl(K(u,v)x,y\bigr)z+\delta L(y,z)K(u,v)x=L\bigl(K(u,v)x,z)y$ and a similar relation holds interchanging $x$ and $z$, whence \eqref{eq:KKLK.}.

On the other hand, equation \eqref{eq:FK1} gives:
\[
\begin{split}
[L(u,v)&,L(x,y)]-\delta [L(u,y),L(x,v)]\\
 &=L\bigl(L(u,v)x,y\bigr)-\delta L\bigl(L(u,y)x,v\bigr)+\epsilon L\bigl(x,K(v,y)u\bigr).
\end{split}
\]
Interchange $u$ and $x$ to get
\[
\begin{split}
[L(x,v)&,L(u,y)]-\delta [L(x,y),L(u,v)]\\
 &=L\bigl(L(x,v)u,y\bigr)-\delta L\bigl(L(x,y)u,v\bigr)+\epsilon L\bigl(u,K(v,y)x\bigr),
\end{split}
\]
and this gives
\[
\begin{split}
\delta L\bigl(L(u,v)x,y\bigr)&-L\bigl(L(u,y)x,v\bigr)+\epsilon\delta L\bigl(x,K(v,y)u\bigr)\\
 &=L\bigl(L(x,v)u,y\bigr)-\delta L\bigl(L(x,y)u,v\bigr)+\epsilon L\bigl(u,K(v,y)x\bigr),
\end{split}
\]
or
\[
\delta L\bigl(x,K(v,y)u\bigr)-L\bigl(u,K(v,y)x\bigr)=\epsilon L\bigl(K(x,u)v,y\bigr)+\epsilon L\bigl(K(u,x)y,v\bigr),
\]
whose right hand side equals $-K(x,u)K(v,y)$, due to \eqref{eq:KKLK.},  thus getting \eqref{eq:KKL.K}.
\end{proof}

\smallskip

Coming back to the case $\epsilon=\delta=1$, for $a_1,a_2,a_3,b_1,b_2,b_3,x\in U$ we have
\[
\begin{split}
&\left[\left[\begin{pmatrix} 0&0\\ K(a_1,b_1)&0\end{pmatrix},\theta\begin{pmatrix} 0&0\\ K(a_2,b_2)&0\end{pmatrix}\right],\begin{pmatrix} 0&0\\ K(a_3,b_3)&0\end{pmatrix}\right]\\[4pt]
 &\quad=\begin{pmatrix} 0&\quad 0\quad\\
 -K(a_1,b_1)K(a_2,b_2)K(a_3,b_3)-K(a_3,b_3)K(a_2,b_2)K(a_1,b_1) &0\end{pmatrix},\\[10pt]
&\left[\left[\begin{pmatrix} 0&0\\ K(a_1,b_1)&0\end{pmatrix},\theta\begin{pmatrix} 0&0\\ K(a_2,b_2)&0\end{pmatrix}\right],\begin{pmatrix} x\\ 0\end{pmatrix}\right]
 =\begin{pmatrix} K(a_2,b_2)K(a_1,b_1)x\\ 0\end{pmatrix}.
\end{split}
\]
Therefore, we get the following instances of the triple product on $A=K(U,U)\oplus U$:
\[
\begin{split}
\{ M_1,M_2,M_3\}&=-\bigl(M_1M_2M_3+M_3M_2M_1\bigr),\\
\{M_1,M_2,x\}&=M_2M_1x,
\end{split}
\]
for any $M_1,M_2,M_3\in K(U,U)$ and $x\in U$. (Note that the triple product $\{M_1,M_2,M_3\}$ is, up to the sign, the usual Jordan triple product in the associative algebra $\End_\bF(U)$.)

Finally, for any $x_1,x_2\in U$:
\[
\left[\begin{pmatrix} x_1\\ 0\end{pmatrix},\theta\begin{pmatrix} x_2\\ 0\end{pmatrix}\right]=\left[\begin{pmatrix} x_1\\ 0\end{pmatrix},\begin{pmatrix} 0\\ x_2\end{pmatrix}\right] =
\begin{pmatrix} L(x_1,x_2)&0\\ 0&\epsilon L(x_2,x_1)\end{pmatrix},
\]
so that we obtain, for any $x_1,x_2,x_3,a,b\in U$:
\[
\begin{split}
&\left[\left[\begin{pmatrix} x_1\\ 0\end{pmatrix},\theta\begin{pmatrix} x_2\\ 0\end{pmatrix}\right],\begin{pmatrix} 0&0\\ K(a,b)&0\end{pmatrix}\right]\\
 &\qquad\qquad\qquad\qquad =\begin{pmatrix} 0&0\\
  L(x_2,x_1)K(a,b)-K(a,b)L(x_1,x_2) &0\end{pmatrix}\\
  &\qquad\qquad\qquad\qquad =\begin{pmatrix} 0&0\\ K\bigl( K(a,b)x_1,x_2\bigr) &0\end{pmatrix} \quad\text{(see \eqref{eq:FK2})}\\[6pt]
&\left[\left[\begin{pmatrix} x_1\\ 0\end{pmatrix},\theta\begin{pmatrix} x_2\\ 0\end{pmatrix}\right],\begin{pmatrix} x_3\\ 0\end{pmatrix}\right]
 =\begin{pmatrix} L(x_1,x_2)x_3\\ 0\end{pmatrix}=\begin{pmatrix} x_1x_2x_3\\ 0\end{pmatrix}.
\end{split}
\]
Hence we get the following instances of the triple product on $A=K(U,U)\oplus U$:
\[
\begin{split}
\{ x_1,x_2,M\}&=K\bigl(Mx_1,x_2\bigr),\\
\{x_1,x_2,x_3\}&=x_1x_2x_3,
\end{split}
\]
for any $x_1,x_2,x_3\in U$ and $M\in K(U,U)$.

\smallskip

Let us summarize the previous arguments in our last result, which attaches a dicyclic ternary algebra to an arbitrary $(1,1)$ Freudenthal-Kantor triple system:

\begin{theorem}\label{th:11FKTSdicyclic}
Let $(U,xyz)$ be a $(1,1)$ Freudenthal-Kantor triple system. Then $\bigl(K(U,U)\oplus U,\bar{\ },*,\{.,.,.\}\bigr)$ is a dicyclic ternary algebra, where
\[
\begin{split}
&\bar M=M,\quad \bar x=-x,\\[2pt]
&M_1*M_2=0,\quad M*x=Mx,\quad x_1*x_2=-K(x_1,x_2),\\[2pt]
&\{M_1,x_1,M_2\}=\{M_1,x_1,x_2\}=\{ x_1,M_1,M_2\}=\{ x_1,M_1,x_2\}=0,\\[2pt]
&\{ M_1,M_2,M_3\}=-\bigl(M_1M_2M_3+M_3M_2M_1\bigr),
   \quad\{M_1,M_2,x\}=M_2M_1x,\\
&\{ x_1,x_2,M\}=K\bigl(Mx_1,x_2\bigr),\quad \{x_1,x_2,x_3\}=x_1x_2x_3,
\end{split}
\]
for any $x,x_1,x_2,x_3\in U$ and $M,M_1,M_2,M_3\in K(U,U)$. \qed
\end{theorem}

\end{document}